\documentclass[12pt]{amsart}
\usepackage{amssymb,enumerate,hyperref,hyphenat}
\usepackage[all]{xy}
\usepackage[toc,page,title,titletoc,header]{appendix}
\usepackage[foot]{amsaddr}
\usepackage{xcolor}

\theoremstyle{plain}
\newtheorem{Thm}{Theorem}[section]
\newtheorem{Prop}[Thm]{Proposition}
\newtheorem{Lem}[Thm]{Lemma}
\newtheorem{Que}[Thm]{Question}
\newtheorem{Cor}[Thm]{Corollary}

\theoremstyle{remark}
\newtheorem{Rem}[Thm]{Remark}

\theoremstyle{definition}
\newtheorem{Def}[Thm]{Definition}
\newtheorem{Eg}[Thm]{Example}
\newtheorem*{ZDO}{Zariski-dense orbit conjecture}

\def\A{\mathbb A}
\def\B{\mathbb B}
\def\Z{\mathbb Z}
\def\Q{\mathbb Q}
\def\R{\mathbb R}
\def\C{\mathbb C}
\def\D{\mathbb D}
\def\P{\mathbb P}
\def\sM{\mathcal M}
\def\sO{\mathcal O}
\def\sL{\mathcal L}
\def\sX{\mathcal X}
\def\sC{\mathcal C}
\def\Om{\Omega}
\def\la{\lambda}
\def\rog{{\rm rog}}
\def\PGL{{\rm PGL}}
\def\Spec{{\rm Spec}}
\def\Zar{{\rm Zar}}
\def\alg{{\rm alg}}
\def\Per{{\rm Per}}
\def\Aut{{\rm Aut}}
\def\Var{{\rm Var}}
\def\Gal{{\rm Gal}}
\def\Rat{{\rm Rat}}
\def\Hom{{\rm Hom}}
\def\PHom{{\rm PHom}}
\def\FL{{\rm FL}}

\begin{document}
\title[]{Space spanned by characteristic exponents}

\author{Zhuchao Ji}
\address{Institute for Theoretical Sciences, Westlake University, Hangzhou 310030, China}
\email{jizhuchao@westlake.edu.cn}

\author{Junyi Xie}
\address{Beijing International Center for Mathematical Research, Peking University, Beijing 100871, China}
\email{xiejunyi@bicmr.pku.edu.cn}

\author{Geng-Rui Zhang}
\address{School of Mathematical Sciences, Peking University, Beijing 100871, China}
\email{grzhang@stu.pku.edu.cn}


\bibliographystyle{plain}

\begin{abstract}
We prove several rigidity results on multiplier spectrum and length spectrum. For example, we show that for every non-exceptional rational map $f:\P^1(\C)\to\P^1(\C)$ of degree $d\geq2$, the $\Q$-vector space generated by all the (finite) characteristic exponents of periodic points of $f$ has infinite dimension. This answers a stronger version of a question by Levy and Tucker. Our result can also be seen as a generalization of recent results of Ji-Xie and of Huguin which proved Milnor's conjecture about rational maps having integer multipliers. We also get a characterization of postcritically finite maps by using their length spectra. Finally as an application of our result, we get a new proof of the Zariski-dense orbit conjecture for endomorphisms on $(\P^1)^N, N\geq 1$. 
\end{abstract}

\maketitle
\tableofcontents

\section{Introduction}
Let $f:\P^1\to\P^1$ be a rational map over $\C$ of degree $d\geq2$.
Our aim is to study the $\Q$-vector space spanned by the characteristic exponents of periodic points of a rational map on $\P^1(\C)$ and prove some rigidity results.

\subsection{Multiplier, length, and characteristic exponent}
Let $z_0\in\P^1(\C)$ be a periodic point of $f$ with exact period $n$. Define $n_f(z_0):=n$. We write $n(z_0)$ for simplicity when the map $f$ is clear. The \emph{multiplier} $\rho_f(z_0)$ of $f$ at $z_0$ is defined to be the differential $df^n(z_0)\in\C$. We write $\rho(z_0)$ for simplicity when the map $f$ is clear. 
The \emph{length} of $f$ at $z_0$ is the absolute value $|\rho_f(z_0)|$. The multiplier and length are invariant under conjugacy. The \emph{characteristic exponent} of $f$ at $z_0$ is defined to be $\chi_f(z_0):=n^{-1}\log\lvert\rho_f(z_0)\rvert$ (when $\rho_f(z_0)\neq0$). 

\medskip

Denote by $\Per(f)(\C)$ the set of all periodic points in $\P^1(\C)$ of $f$ and define $\Per^*(f)(\C):=\{z_0\in\Per(f)(\C):\rho_f(z_0)\neq0\}$. When the base field $\C$ is clear, we also write $\Per(f)$ and $\Per^*(f)$ for simplicity.

\medskip

The \emph{Lyapunov exponent} (with respect to the maximal entropy measure) of $f$ is defined by $$\sL_f:=\int_{\P^1(\C)}\log|d f| d\mu_f,$$
where $\mu_f$ is the unique maximal entropy measure of $f$ on $\P^1(\C)$ which is also the canonical measure of $f$ on $\P^1(\C)$ (see \cite{Lyubich83,FLM83,Mane83,Zdunik90}), and the norm of the differential is computed with respect to the spherical metric. In fact, the value of the Lyapunov exponent does not change with the spherical metric replaced by any K{\"a}hler metric on $\P^1_\C$ in the definition. It is known that $\sL_f\geq\log(d)/2$, and the equality holds if and only if $f$ is a Latt\`es map in the sense of Definition \ref{excep} (see \cite{Zdunik90,zdunik2014characteristic}). The quantity $\exp(\sL_f)$ is considered as an average of $|d f|$ on $\P^1(\C)$, which measures the average expansion rate of $f$ along a typical orbit.

\subsection{Exceptional maps}
In complex dynamics, the exceptional maps defined below are often considered as exceptional examples among all rational maps of degree $\geq2$. We may view them as rational maps on $\P^1(\C)$ related to algebraic groups.

\begin{Def}\label{excep}
Let $f:\P^1\to\P^1$ be an endomorphism over $\C$ of degree $d\geq 2$. 
\begin{itemize}
	\item The map $f$ is called \emph{Latt\`es} if there is an endomorphism $\phi$ on an elliptic curve $E$ such that $\phi$ is semi-conjugate to $f$. Here, for an endomorphism $g$ on an algebraic variety $X$ and an endomorphism $h$ on an algebraic variety $Y$, we say $h$ is \emph{semi-conjugate} to $g$ if there is a dominant morphism $\pi:X\to Y$ such that $\pi\circ g=h\circ\pi$. Furthermore, the map $f$ is called \emph{flexible Latt\`es} if there are an elliptic curve $E$ and $n\in\Z\setminus\{0,\pm1\}$ such that $f$ is semi-conjugate to the multiplication-by-$n$ map $[n]$ on $E$ via the quotient map modulo $\{\pm 1\}$ on $E$. Otherwise, it is called \emph{rigid Latt\`es}.
	\item We say that $f$ is of \emph{monomial type} if it is semi-conjugate to the map $z\mapsto z^n$ on $\P^1$ for some integer $n\neq0,\pm1$.
\end{itemize}
We call $f$ \emph{exceptional} if it is Latt\`es or of monomial type. It is well-known that $f$ is exceptional if and only if some iterate $f^k$ is exceptional ($k\in\Z_{>0}$).
\end{Def}

\subsection{Statement of the main results}
We fix an embedding of the algebraic closure $\overline{\Q}$ of $\Q$ in $\C$ and consider all number fields as subfields of $\C$. Denote the usual absolute value on $\C$ by $| \cdot |$.

\medskip

Our first result shows that the definition field of a rational map is determined by its length spectrum, provided that it is not a flexible Latt\`es map. 
\begin{Thm}\label{thmalglenintro}Let $f:\P_\C^1\to\P_\C^1$ be a rational map of degree at least $2$.
Assume that $f$ is not a flexible Latt\`es map and for every $x\in\Per(f)(\C)$, $|\rho_f(x)|\in\overline{\Q}$. Then $f$ is defined over $\overline{\Q}$.
\end{Thm}
\begin{Rem}\label{remcofinitethm1.2}
	In Theorem \ref{thmalglen}, we prove a more general version of Theorem \ref{thmalglenintro}, in which $\overline{\Q}$ can be replaced by any algebraically closed subfield of $\C$ invariant under complex conjugation. In fact, the actual proof shows that we can weaken the conditions in Theorems \ref{thmalglenintro} and \ref{thmalglen} by discarding the corresponding condition for finitely many periodic points of $f$; see Remark \ref{remcofinitesec2}. For example, in Theorem \ref{thmalglenintro}, we can require only that $|\rho_f(x)|\in\overline{\Q}$ for all but finitely many $x\in\Per(f)(\C)$, and the conclusion still holds.
\end{Rem}

McMullen's rigidity of multiplier spectrum \cite[Corollary 2.3]{McMullen1987} asserts that multipliers at all periodic points determine the conjugacy class of rational maps up to only finitely many choices, except for the flexible Latt\`es family. With a standard spread-out argument, McMullen's rigidity theorem implies that, for a rational map $f$ of degree at least $2$ which is not flexible Latt\`es, if its multipliers at periodic points are all algebraic, then $f$ is defined over $\overline{\Q}$. We sketch the argument as follows. Let $f$ be a rational map over $\C$ of degree $d\geq2$ that is not flexible Latt\`es with all multipliers in $\overline{\Q}$. Noetherianity implies that the set $D$ of conjugacy classes $[g]$ (in the moduli space $\sM_d$) such that $g$ has the same multiplier spectrum (i.e., all multipliers at periodic points) as $f$ is an algebraic subset of $\sM_d$ defined over $\overline{\Q}$, since the multiplier spectrum morphism is defined over $\Q$ and all multipliers of $f$ are in $\overline{\Q}$. Then McMullen's rigidity implies that $\dim(D)=0$, hence $D$ is a set of finitely many $\overline{\Q}$-points. Thus $f$ is defined over $\overline{\Q}$ as $[f]\in D$.

Theorem \ref{thmalglenintro} is a generalization of this result from multiplier spectrum to length spectrum (which contains less information). The rigidity of length spectrum was proved by Ji-Xie \cite[Theorem 1.5]{Ji2023}. However, the spread-out argument does not apply directly in this case as the length spectrum map (and its square) is not algebraic on the moduli space of rational maps. Indeed, as shown in \cite[\S 8.1]{Ji2023}, its square is not even real algebraic (though it is still semialgebraic). In \S \ref{sectiontranpoint}, we introduce a method to do the spread-out argument respecting the structure of $\sM_d(\C)$ as a real algebraic variety using Weil restriction. Another difficulty in the length spectrum case is the lack of noetherianity for semialgebraic subsets. We overcome this difficulty by using the notion of admissible subsets introduced in \cite{Ji2023}.

\medskip

The following two results concern the $\Q$-vector space spanned by all the (finite) characteristic exponents of periodic points. 
\begin{Thm}\label{thmmaindim}
	Let $f:\P^1(\C)\to\P^1(\C)$ be a rational map of degree $d\geq2$. Suppose that $f$ is not exceptional. Then the $\Q$-vector space generated by $\{\chi_f(z):z\in\Per^*(f)\}$ in $\R$ has infinite dimension.
\end{Thm}

\medskip

\begin{Thm}\label{thmmainnumf}
	Let $f:\P^1(\C)\to\P^1(\C)$ be a rational map of degree $d\geq2$. Assume that there exists a number field $K$ such that 
	\begin{equation}\label{equassnorm}
		\forall z_0\in\Per(f),\ \exists
		n=n(z_0)\in\Z_{>0},\ \lvert\rho_f(z_0)\rvert^n\in K.
	\end{equation} Then $f$ is exceptional.
\end{Thm}
\begin{Rem}\label{remcofinitethm1.5}
	With the help of Remark \ref{remcofinitethm1.2}, the proof of Theorem \ref{thmmainnumf} in \S \ref{section4} indeed shows that we can weaken the condition of Theorem \ref{thmmainnumf} by requiring only that there is a number field $K$ such that for all but finitely many $z_0\in\Per(f)$, we have $\lvert\rho_f(z_0)\rvert^{n(z_0)}\in K$ for some $n(z_0)\in\Z_{>0}$. 
\end{Rem}

Finitely many nonzero elements $z_1,\dots,z_N$ in a field $F$ are called multiplicatively independent if for every $N$-tuple $(m_1,\dots,m_N)$ of integers, $z_1^{m_1}\cdots z_N^{m_N}=1$ if and only if $m_1=\cdots=m_N=0$. A sequence $(z_n)_{n=1}^\infty$ in $F\setminus\{0\}$ is called multiplicatively independent if all its finite subsequences are multiplicatively independent. Theorem \ref{thmmaindim} immediately implies the existence of infinitely many multipliers for a non-exceptional $f$ whose absolute values are multiplicatively independent.

\begin{Cor}\label{cor1.6}
	Let $f:\P^1(\C)\to\P^1(\C)$ be a rational map of degree $d\geq2$. Suppose that $f$ is not exceptional. Then there exists a sequence $(x_j)_{j=1}^\infty$ in $\Per^*(f)$ such that the sequence $(\lvert\rho_f(x_j)\rvert)_{j=1}^\infty$ is multiplicatively independent in $\R$.
\end{Cor}

\subsection{Motivations and previous results}
\subsubsection{Milnor's conjecture}
Milnor \cite{milnor2006lattes} showed that an exceptional rational map $f:\P^1(\C)\to\P^1(\C)$ of degree $d\geq2$ must have all its multipliers of periodic points in the ring of integers $\sO_K$ for some imaginary quadratic number field $K$, and in fact in $\Z$ when $f$ is not a rigid Latt\`es map. Milnor conjectured that the converse is also true. Milnor's conjecture was recently proved by Ji-Xie:
\begin{Thm}[{\cite[Theorem 1.13]{Ji2023}}]\label{thmjx}
Let $f:\P^1(\C)\to\P^1(\C)$ be a rational map of degree $d\geq2$. Assume that there exists an imaginary quadratic field $K$ such that all multipliers of $f$ belong to $\sO_K$. Then $f$ is exceptional.
\end{Thm}
See also \cite{Buff2022} for a different proof. Recently, Huguin generalized the above result using a different approach:
\begin{Thm}[{\cite[Theorem 7]{Huguin2023}}]\label{thmh}
Let $f:\P^1(\C)\to\P^1(\C)$ be a rational map of degree $d\geq2$. Assume that there exists a number field $K$ such that all multipliers of $f$ belong to $K$. Then $f$ is exceptional.
\end{Thm}
Since our assumption \eqref{equassnorm} in Theorem \ref{thmmainnumf} is weaker than that of Theorem \ref{thmh}, Theorem \ref{thmmainnumf} is a generalization of Theorem \ref{thmh}. Note that our assumption \eqref{equassnorm} is even weaker than the condition that there is a number field $K$ such that
\begin{equation}
	\forall z_0\in\Per(f),\ \exists
	n\in\Z_{>0},\ (\rho_f(z_0))^n\in K.
\end{equation}

\subsubsection{A question of Levy and Tucker}
On the other hand, at the 2014 AIM workshop \emph{Postcritically Finite Maps In Complex And Arithmetic Dynamics}, Levy \cite{Levy2014} and Tucker \cite{Tucker2014} independently asked the following question: 
\begin{Que}
Let $f:\P^1(\C)\to\P^1(\C)$ be a non-exceptional rational map of degree $d\geq2$, and let $S$ be the set of all multipliers of periodic points of $f$. Take the subgroup of $\C^\ast$ generated by $S\setminus\{0\}$. Is it true that the rank of this group is infinite? 
\end{Que}

It is not hard to see that our Corollary \ref{cor1.6} gives a positive answer to (a generalized version of) Levy and Tucker's question. 

\subsection{Sketch of the proofs}
We have already given the idea of Theorem \ref{thmalglenintro}. Here we explain the proofs of Theorem \ref{thmmaindim} and Theorem \ref{thmmainnumf}.

\medskip

We first give the idea of the proof of Theorem \ref{thmmainnumf}. We argue by contradiction and suppose that $f$ is not exceptional. The first step is to reduce to the case where $f$ is defined over $\overline{\Q}$, which can be done using our Theorem \ref{thmalglenintro}. After enlarging $K$, we may assume that both $f$ and $\overline{f}$ are defined over $K$, where $\overline{f}$ is the rational map obtained from $f$ by replacing the coefficients with their complex conjugates. In the second step, we combine an arithmetic equidistribution theorem with a result of Zdunik \cite{zdunik2014characteristic} (see \cite{Huguin2023}) on the Lyapunov exponent to get a contradiction. This argument is inspired by Huguin's proof of Theorem \ref{thmh}. Contrary to the case of Theorem \ref{thmh}, we cannot apply the equidistribution theorem to the one-dimensional dynamical system $f:\P^1\to\P^1$ directly. Our idea is to consider the two-dimensional system $F:=f\times\overline{f}$ on $\P^1\times\P^1$ instead. More precisely, applying a result of Zdunik \cite{zdunik2014characteristic}, we get a sequence $(x_n)_{n=1}^\infty$ of distinct periodic points of $f$ such that
$$\lim\limits_{n\to+\infty}\chi_f(x_n)=a>\sL_f.$$
Consider the endomorphism $F:=f\times\overline{f}$ on $\P^1\times\P^1$ and $\Gamma:=\overline{\{p_n=(x_n,\overline{x_n})\}}^\Zar\subseteq\P^1\times\P^1$. Without loss of generality, we may assume that $\Gamma$ is irreducible and that the sequence $(p_n)_{n=1}^\infty$ is generic in $\Gamma$. Since the Dynamical Manin-Mumford problem is solved for such an endomorphism $F$ by \cite{Ghioca2018c} and $\Gamma$ contains a Zariski-dense subset $\{p_n:n\geq1\}$ of $F$-periodic points, the subscheme $\Gamma$ itself is $F$-periodic. Hence, after replacing $F$ by a suitable iteration, we may assume that $\Gamma$ is $F$-invariant.

Let $\nu_n$ be the discrete probability measure equally supported at the union of Galois orbits of iterates of $p_n$ under $F$. Then $\nu_n$ converges weakly to the canonical measure $\mu$ on $\Gamma$ with respect to $F$ by an equidistribution-type theorem (Theorem \ref{equid}), which is a reformulation of \cite[Theorem 3.1]{Yuan2008}; see \S \ref{sectionequid} for details. Applying $\nu_n\to\mu$ to the continuous test functions $\max\{\log\lvert\det(dF)\rvert,A\}$ ($A\in\R$) and letting $A\to-\infty$, we get
$$2a\leq\int\log\lvert\det(dF)\rvert d\mu,$$
which is impossible since the right-hand side equals $2\sL_f<2a$ by a direct computation. 

\medskip

Next, we sketch the proof of Theorem \ref{thmmaindim}. According to \cite{Douady1993}, postcritically finite (PCF) maps are defined over $\overline{\Q}$ in the moduli space $\sM_d$ of rational maps of degree $d$, except for the family of flexible Latt\`es maps. So it suffices to consider the following two cases: 1). $f$ is defined over $\overline{\Q}$, and 2). $f$ is not PCF. For the first case, the conclusion follows from Theorem \ref{thmmainnumf}. For the second case, we need to develop some new techniques, which are presented in \S \ref{sectionlinear}. In \S \ref{sectionlinear}, we consider some pseudo-linear algebra (meaning that the domain of a map may not be the whole vector space) and the vector space $\D(k)_\Q=k^*\otimes_{\Z}\Q$ for a field $k$ of characteristic zero. We will actually prove a theorem (Theorem \ref{thminvo}) stronger than the non-PCF case of Theorem \ref{thmmaindim}; see \S \ref{sectionlinear} and \S \ref{section6} for details. To prove Theorem \ref{thminvo}, in \S \ref{section6.1} we first deal with the case that $f$ is defined over $\overline{\Q}$. A key ingredient in this step is \cite[Lemma 4.1]{Benedetto2012}, which is a consequence of Siegel's theorem on $S$-integral points. The existence of a non-preperiodic critical point of $f$ is essentially used here. In \S \ref{section6.2}, we consider the general case and complete the proof. This is achieved by reducing to the case that $f$ is defined over $\overline{\Q}$ via an algebraic-geometric argument and the techniques developed in \S \ref{sectionlinear}.

\subsection{Applications}
\subsubsection{The Zariski-dense orbit conjecture}
By applying Corollary \ref{cor1.6}, we can give a new proof of a special case of the Zariski-dense orbit conjecture.
\begin{ZDO}[=ZDO]
Let $k$ be an algebraically closed field of characteristic $0$. Given an irreducible quasi-projective variety $X$ over $k$ and a dominant rational self-map $f$ on $X$. If $\{g\in k(X):g\circ f=g\}=k$, where $k(X)$ is the function field of $X$, then there exists $x\in X(k)$ whose forward orbit under $f$ is well-defined and Zariski-dense in $X$.
\end{ZDO}

\begin{Rem}
The converse of ZDO holds and is easy to prove. For some progress on ZDO, see for example \cite{Amerik2011,Amerik2008,Medvdev,Xie2017,Xie2022}.
\end{Rem}
As an application of Corollary \ref{cor1.6}, we give a new proof of (the most difficult part of) a special case of ZDO, which was first proved by Xie \cite[Theorem 1.16]{Xie2022}. 
\begin{Thm}\label{thm1.8}
Let $X=\P^1_k\times\cdots\times\P^1_k$ be the product of $N$ copies of the projective line over an algebraically closed field $k$ of characteristic $0$. Suppose that $f:X\to X$ is an endomorphism of the form $f_1\times\cdots\times f_N$, where $f_j:\P^1_k\to\P^1_k$ is a non-constant rational map for $1\leq j\leq N$. Then the ZDO holds for $(X,f)$.
\end{Thm}
\begin{Rem}
We note that every dominant endomorphism $f:(\P^1)^N\to(\P^1)^N$ over an algebraically closed field $k$ of characteristic zero must have the form $f_1\times\cdots\times f_N$, after replacing $f$ by a suitable iterate. 
\end{Rem}

The original proof of Theorem \ref{thm1.8} in \cite{Xie2022} relies on the solution of the (adelic) Zariski-dense orbit conjecture for smooth projective surfaces \cite[Theorem 1.15]{Xie2022}, the notion of adelic topology introduced in \cite[\S 3]{Xie2022}, and a classification result \cite[Proposition 9.2]{Xie2022} on invariant subvarieties of $f:(\P^1)^N\to(\P^1)^N$ (see also \cite{Medvdev,Ghioca2018c}). When $N=2$, Pakovich gave another proof \cite{Pakovich2023} of Theorem \ref{thm1.8} using his classification of invariant curves in $\P^1\times\P^1$ and a height argument. We give a new proof in the case where all the $f_i$'s are non-exceptional of degree at least $2$; this proof does not require the ingredients mentioned above. Other cases of Theorem \ref{thm1.8} are relatively easy or can be reduced to the case where all $f_i$ are non-exceptional, as shown in \cite[\S 9.3]{Xie2022}. 

\subsubsection{A characterization of PCF maps}
We also show that one can decide whether a rational map $f:\P^1(\C)\to\P^1(\C)$ of degree $d\geq2$ is PCF with the information of its multiplier spectrum or length spectrum on periodic points.

\begin{Thm}\label{pcfdec}
Let $f:\P^1(\C)\to\P^1(\C)$ be a rational map of degree $d\geq2$. Then the following statements are equivalent:

(1) $f$ is PCF;

(2) $\rho_f(x)\in\overline{\Q}$ for all $x\in\Per(f)(\C)$ and the $\Q$-subspace $V=V(f)$ of $\R$ is of finite dimension, where $V$ is generated over $\Q$ by $$\{\log\lvert N_{K_x/\Q}(\rho_f(x))\rvert:x\in\Per^*(f)(\C)\};$$

(3) $\lvert\rho_f(x)\rvert\in\overline{\Q}$ for all $x\in\Per(f)(\C)$ and the $\Q$-subspace $W=W(f)$ of $\R$ is of finite dimension, where $W$ is generated over $\Q$ by $$\{\log\lvert N_{L_x/\Q}(\lvert\rho_f(x)\rvert)\rvert:x\in\Per^*(f)(\C)\}.$$

Here $K_x$ (resp. $L_x$) is any number field containing $\rho_f(x)$ (resp. $\lvert\rho_f(x)\rvert$), and $N_{K_x/\Q}$ (resp. $N_{L_x/\Q}$) is the norm map for the extension $K_x/\Q$ (resp. $L_x/\Q$), i.e., the determinant of the $\Q$-linear transformation induced by multiplication by $\rho_f(x)$ (resp. $|\rho_f(x)|$).

Clearly, the subspaces $V$ and $W$ above are independent of the choices of the fields $K_x$ and $L_x$, respectively. We also remark that the $\Q$-vector spaces $V$ in (2) and $W$ in (3) are the same. 
\end{Thm}
\begin{Rem}
	Similarly, the actual proof shows that we can weaken conditions (2) and (3) in Theorem \ref{pcfdec}, in the spirit of Remarks \ref{remcofinitethm1.2} and \ref{remcofinitethm1.5}. Precisely, we can add two more equivalent statements to Theorem \ref{pcfdec}:
	
	$(2)^*$ $\rho_f(x)\in\overline{\Q}$ for all but finitely many $x\in\Per(f)(\C)$ and the $\Q$-subspace $V^*=V^*(f)$ of $\R$ is of finite dimension, where $V^*$ is generated over $\Q$ by $$\{\log\lvert N_{K_x/\Q}(\rho_f(x))\rvert:x\in\Per^*(f)(\C),\rho_f(x)\in\overline{\Q}\};$$
	
	$(3)^*$ $\lvert\rho_f(x)\rvert\in\overline{\Q}$ for all but finitely many $x\in\Per(f)(\C)$ and the $\Q$-subspace $W^*=W^*(f)$ of $\R$ is of finite dimension, where $W^*$ is generated over $\Q$ by $$\{\log\lvert N_{L_x/\Q}(\lvert\rho_f(x)\rvert)\rvert:x\in\Per^*(f)(\C),\lvert\rho_f(x)\rvert\in\overline{\Q}\}.$$
\end{Rem}
The proofs of Theorems \ref{thm1.8} and \ref{pcfdec} will be given in \S \ref{section6}.

\section{Rational maps with algebraic lengths}
Let $K$ be an algebraically closed subfield of $\C$ which is invariant under complex conjugation $\tau$, i.e., $\tau(K)=K$. The aim of this section is to prove the following theorem.

\begin{Thm}\label{thmalglen}
Let $f:\P_\C^1\to\P_\C^1$ be a rational map of degree $d\geq 2$. Assume that $f$ is not a flexible Latt\`es map and that for every $x\in\Per(f)(\C)$, $\lvert\rho_f(x)\rvert\in K$. Then $f$ is defined over $K$.
\end{Thm}

\begin{Rem}\label{remcofinitesec2}
The condition in Theorem \ref{thmalglen} can be weakened: it suffices to require that $f$ is not flexible Latt\`es and $\lvert\rho_f(x)\rvert\in K$ for all but finitely many $x\in\Per(f)(\C)$. To show this, we need to replace $A$ in the proof in \S \ref{sectionalglenpf} by $A=(A_n=(RL^*(f)_n)\cap K)_n$, and then the original proof applies. Here, $A_n=(RL^*(f)_n)\cap K$ is the sub-multiset of the multiset $RL^*(f)_n$ consisting of elements that also lie in $K$; see \S \ref{sectionlenrig} and \S \ref{sectionalglenpf}.
\end{Rem}

Applying Theorem \ref{thmalglen} to the case $K=\overline{\Q}$, we get Theorem \ref{thmalglenintro}. We introduce the notions of Weil restriction, admissible subsets, and transcendental points in \S \ref{sectionweilrest}--\ref{sectiontranpoint}. Then we construct certain moduli spaces and study the length map on them in \S \ref{sectionmoduli}--\ref{sectionlenrig}. After these preparations, we complete the proof of Theorem \ref{thmalglen} in \S \ref{sectionalglenpf}.

\subsection{Weil restriction}\label{sectionweilrest}
Recall that $K$ is an algebraically closed subfield of $\C$ such that $\tau(K)=K$. Set $L:=K^\tau=K\cap\R$. For example, if $K=\C$, then $L=\R$. We need the following easy lemma.
\begin{Lem}\label{lemkoltwo}
We have $K=L+iL$; in particular $[K:L]=2$. 
\end{Lem}
\begin{proof}[Proof of Lemma \ref{lemkoltwo}]
	Since $K$ is algebraically closed, $i\in K$. In particular, $K\neq L$. For every $u\in K$, we may write
	$$u=\frac{u+\tau(u)}{2}+\frac{u-\tau(u)}{2i}i$$
	and both $\frac{u+\tau(u)}{2}$ and $\frac{u-\tau(u)}{2i}$ are contained in $L$. This completes the proof.
\end{proof}

To prove Theorem \ref{thmalglen}, we need to analyze the length map while respecting the structure of $\sM_d(K)$ as an algebraic variety over $L$ using Weil restriction, due to the appearance of the absolute value in the condition. Hence, the notion of Weil restriction, which we will briefly recall, is necessary. See \cite[\S 4.6]{Poonen2017} and \cite[\S 7.6]{Bosch1990} for more details. 

\medskip

Denote by $\Var_{/K}$ (resp. $\Var_{/L}$) the category of varieties over $K$ (resp. $L$). For every variety $X$ over $K$, there is a unique variety $R(X)$ over $L$ which represents the functor $\Var_{/L}\to {\rm Sets}$ sending $V\in\Var_{/L}$ to $\Hom(V\otimes_L K,X)$. It is called the \emph{Weil restriction of $X$ with respect to $K/L$}. The functor $X\mapsto R(X)$ is called the Weil restriction. There is a canonical bijection $\psi_K: X(K)\to R(X)(L)$ between sets. When $K=\C$, this map is a real analytic diffeomorphism. One may view $X(K)$ as the set of $L$-points of an $L$-algebraic variety $R(X)$ via $\psi_K$.

\begin{Def}\label{defirealzt}
The \emph{$L$-Zariski topology} on $X(K)$ is the restriction of the Zariski topology on $R(X)$ to $R(X)(L)$ via the identification $\psi_K$. A subset $Y$ of $X(K)$ is called \emph{$L$-algebraic} if it is closed in the $L$-Zariski topology. When $K=\C$, the $L$-Zariski topology coincides with the real Zariski topology as in \cite[\S 8.1.1]{Ji2023}.
\end{Def}

By the last statement of Proposition \ref{probasicweil} below, the $L$-Zariski topology is finer than the Zariski topology on $X(K)$.

\medskip

When $K=\C$, roughly speaking, the Weil restriction is just constructed by splitting a complex variable $z$ into two real variables $x,y$ via $z=x+iy$. For the convenience of the reader, in the following example, we show the concrete construction of $R(X)$ when $X$ is affine. 

\begin{Eg}
First assume that $X=\A^N_K$. Then $R(X)=\A^{2N}_L$. The map $$\psi_K:\A^N_K(K)=K^N\to\A^{2N}_L(L)=L^{2N}$$ sends $(z_1,\dots,z_N)$ to $(x_1,y_1,x_2,y_2,\dots,x_N,y_N)$, where $z_j=x_j+iy_j$ for $1\leq j\leq N$.

Consider the algebra
\begin{equation}\label{eqB}
	\B:=L[I]/(I^2+1)\simeq L\oplus IL\simeq K
\end{equation}
by Lemma \ref{lemkoltwo}, where the variable $I$ satisfying $I^2=-1$ corresponds to $i$ in $z_j=x_j+iy_j$. Via \eqref{eqB}, every $f\in K[z_1,\dots,z_N]\simeq\B[z_1,\dots,z_N]$ defines an element
$$F:=f(x_1+Iy_1,\dots,x_N+Iy_N)\in\B[x_1,y_1,\dots,x_N,y_N].$$
Since
$$\B[x_1,y_1,\dots,x_N,y_N]=L[x_1,y_1,\dots,x_N,y_N]\oplus IL[x_1,y_1,\dots,x_N,y_N],$$ $F$ can be uniquely decomposed as $$F=r(f)+Ii(f)$$ where $r(f),i(f)\in L[x_1,y_1,\dots,x_N,y_N]$ are determined by $f$.

More generally, if $X$ is the closed subvariety of $\A^N_K=\Spec K[z_1,\dots,z_N]$ defined by the ideal $(f_1,\dots,f_s)$ of $K[z_1,\dots,z_N]$, then $R(X)$ is the closed subvariety of
$$R(\A^N_K)=\A^{2N}_L=\Spec L[x_1,y_1,\dots,x_N,y_N]$$
defined by the ideal $(r(f_1),i(f_1),\dots,r(f_s),i(f_s))$ of $L[x_1,y_1,\dots,x_N,y_N]$.
\end{Eg}

\medskip

We list some basic properties of Weil restriction without proof.
\begin{Prop}\label{probasicweil}
Let $X,Y\in\Var_{/K}$. The following properties hold:
\begin{itemize}
\item if $X$ is irreducible, then $R(X)$ is irreducible;
\item $\dim R(X)=2\dim X$;
\item if $f:Y\to X$ is a closed (resp. open) immersion, then the induced morphism $R(f):R(Y)\to R(X)$ is a closed (resp. open) immersion.
\end{itemize}
\end{Prop}

We still denote by $\tau$ the restriction of $\tau$ to $K$. Denote by $X^\tau$ the base change of $X$ by the homomorphism $\tau:K\to K$. This induces a morphism of schemes (over $\Z$) $\tau:X^\tau\to X$. It is not a morphism of schemes over $K$. It is clear that $(X^\tau)^\tau=X$.
\begin{Eg}
If $X$ is the subvariety of $\A^N_K=\Spec K[z_1,\dots,z_N]$ defined by the equations $\sum_I a_{i,I}z^I=0$ ($i=1,\dots,s$), then $X^\tau$ is the subvariety of $\A^N_K$ defined by $\sum_I\tau(a_{i,I})z^I=0$ ($i=1,\dots,s$). The map $\tau: X=(X^\tau)^\tau\to X^\tau$ sends a point $(z_1,\dots,z_N)\in X(K)$ to $(\tau(z_1),\dots,\tau(z_N))\in X^\tau(K)$.
\end{Eg}

The following result due to Weil is useful for computing the Weil restriction.
\begin{Prop}[{\cite[Exercise 4.7]{Poonen2017}}]\label{proweilconj}
There is a canonical isomorphism
$$R(X)\otimes_L K\simeq X\times X^\tau.$$
Under this isomorphism,
$$R(X)(L)=\{(z_1,z_2)\in X(K)\times X^{\tau}(K)\mid z_2=\tau(z_1)\}$$
and $\psi_K$ sends $z\in X(K)$ to $(z,\tau(z))\in R(X)(L)$.
\end{Prop}

\subsection{Admissible subsets}
In this subsection, we recall the notion of admissible subsets of real algebraic varieties introduced in \cite{Ji2023}, which satisfies the important descending chain condition.

\medskip

Let $X$ be a variety over $\R$.
\begin{Def}[{\cite[\S 8.2]{Ji2023}}]
A closed subset $V$ of $X(\R)$ (with the real topology) is called \emph{admissible} if there is a morphism $f: Y\to X$ of real algebraic varieties and a Zariski closed subset $V^\prime\subseteq Y$ such that $V=f(V^\prime(\R))$ and $f$ is \'etale at every point in $V^\prime(\R)$.
\end{Def}

In particular, every algebraic subset of $X(\R)$ is admissible.

\begin{Rem}
Denote by $J$ the non-\'etale locus of $f$ in $Y$. Then $J\cap V^\prime(\R)=\emptyset$. Since we may replace $(Y,V^\prime)$ by $(Y\setminus J,V^\prime\cap(Y\setminus J))$, in the above definition we may further assume that $f$ is \'etale.
\end{Rem}

\begin{Prop}[{\cite[Remarks 8.14, 8.15, and Proposition 8.16]{Ji2023}}]\label{propadbasic}
The following basic properties hold:
\begin{itemize}
\item[(1)] Let $Y$ be a Zariski closed subset of $X$.
If $V$ is admissible as a subset of $X(\R)$, then $V\cap Y$ is admissible as a subset of $Y(\R)$.
\item[(2)] An admissible subset is semialgebraic. 
\item[(3)] Let $V_1$ and $V_2$ be two admissible closed subsets of $X(\R)$. Then $V_1\cap V_2$ is admissible.
\end{itemize}
\end{Prop}

The following theorem shows that admissible subsets satisfy the descending chain condition.
\begin{Thm}[{\cite[Theorem 8.17]{Ji2023}}]\label{thmNoetherianad}
Let $(V_n)_{n\geq 1}$ be a decreasing sequence of admissible subsets of $X(\R)$. Then there exists $N\geq 1$ such that $V_n=V_N$ for all $n\geq N$.
\end{Thm}

\subsection{Transcendental points}\label{sectiontranpoint}
Let $X_K$ be a variety over $K$ and $X:= X_K\otimes_K\C$. We regard $X_K$ as a model of $X$ over $K$.

Denote by $\pi_K: X\to X_K$ the natural projection. For every point $x\in X(\C)$, define $Z(x)_K$ to be the Zariski closure of $\pi_K(x)$ in $X_K$ and $Z(x):=\pi_K^{-1}(Z(x)_K)\subseteq X$. It is clear that $Z(x)$ is irreducible. We call $Z(x)$ the \emph{$\C/K$-closure of $x$ with respect to the model $X_K$}. We say that $x$ is \emph{transcendental} if $\dim Z(x)\geq 1$ and call $\dim Z(x)$ the \emph{transcendence degree} of $x$.

\medskip

The notion of transcendental points plays an important role in the proof of some cases of the geometric Bombieri-Lang conjecture by Xie-Yuan \cite{Partial1,Partial2} and in the proof of the dynamical Andr\'e-Oort conjecture for curves by Ji-Xie \cite{ji2023dao}. Roughly speaking, a very general point in $Z(x)$ satisfies the same algebraic properties as $x$. In this paper, we study lengths of periodic points, whose definition (see \S1.1) involves the absolute value function $|\cdot|:\C\to\R_{\geq 0}$ which is not algebraic. However, the square $|\cdot|^2:\C\to\R$ is real algebraic. For this reason, we need to generalize the notion of transcendental points to respect the structure of $X_K(K)$ as an algebraic variety over $L$ and the structure of $X(\C)$ as a real algebraic variety. This generalized notion is important in the proof of Theorem \ref{thmalglen}. 

\medskip

The Weil restriction $R(X)$ of $X$ with respect to $\C/\R$ is an algebraic variety over $\R$, and the Weil restriction $R(X_K)$ of $X_K$ with respect to $K/L$ is an algebraic variety over $L$. We have $R(X)=R(X_K)\otimes_L\R$. Denote by $\pi_L:R(X)\to R(X_K)$ the natural projection. For every $x\in X(\C)$, let $Y(x)_L$ be the Zariski closure of $\pi_L(\psi_\C(x))$ in $R(X_K)$ and $Y(x):=\pi_L^{-1}(Y(x)_L)\subseteq R(X)$, where $\psi_\C:X(\C)\to R(X)(\R)$ is the identification from \S \ref{sectionweilrest}. Set $Z^{\R}(x):=\psi_\C^{-1}(Y(x)(\R))$, which is a real Zariski closed subset of $X(\C)$.

\medskip

We now give a more concrete description of $Z(x)$ and $Z^{\R}(x)$. Let $U_K$ be an irreducible affine open neighborhood of $\pi_K(x)$ in $X_K$ (over $K$). Set $U:=\pi^{-1}_K(U_K)=U_K\otimes_K\C$. We have a natural embedding $\pi_K^*:\sO(U_K)\hookrightarrow\sO(U)$. We view the elements in $\sO^K(U):=\pi_K^*(\sO(U_K))\subseteq\sO(U)$ as the algebraic functions on $U(\C)$ defined over $K$. Then we have
$$Z(x)\cap U=\{y\in U\mid h(y)=0\text{ for every }h\in\sO^K(U)\text{ with }h(x)=0\}$$
and $Z(x)$ is the Zariski closure of $Z(x)\cap U$.

\medskip

As $\sO(R(U)_\C)=\sO(R(U))\otimes_\R\C$, every $h\in\sO(R(U)_\C)$ can be viewed as a $\C$-valued algebraic function on $R(U)(\R)$. Every $h\in\sO(R(U)_\C)$ induces a function $h\circ\psi_\C$ on $U(\C)$. Functions of this form are exactly the $\C$-valued real algebraic functions on $U(\C)$. Denote by $\sC^{\R-\alg}(U)$ the $\R$-algebra of $\C$-valued real algebraic functions on $U(\C)$. Since algebraic functions are real algebraic, we have a natural embedding $\sO(U)\subseteq\sC^{\R-\alg}(U)$. By Proposition \ref{proweilconj}, we have
$$\sC^{\R-\alg}(U)\simeq\sO(U)\otimes_\C\tau(\sO(U)).$$

Let $\sO^L(R(U)):=\pi_L^*(\sO(R(U_K)))$ be the set of algebraic functions defined over $L$ on $R(U)$. Let $\sC^{\R-\alg,L}(U)$ be the image of $\sO^L(R(U))\otimes_L K$ in $\sC^{\R-\alg}(U)$, which is the set of $\C$-valued real algebraic functions on $U(\C)$ defined over $L$. It is clear that $\sO^K(U)\subseteq\sC^{\R-\alg, L}(U)$. By Proposition \ref{proweilconj}, we have
$$\sC^{\R-\alg,L}(U)\simeq\sO^K(U)\otimes_K\tau(\sO^K(U)).$$

We have
$$Z^{\R}(x)\cap U(\C)=\{y\in U(\C)\mid h(y)=0\text{ for every } h\in\sC^{\R-\alg,L}(U)\text{ with }h(x)=0\}$$
and $Z^{\R}(x)$ is the real Zariski closure of $Z^{\R}(x)\cap U(\C)$. This implies the following lemma.
\begin{Lem}\label{lemmorphismzx}
Let $f_K: X^\prime_K\to X_K$ be a morphism between $K$-varieties. Set $X^\prime:=X^\prime_K\otimes_K\C$ and let $f: X^\prime\to X$ be the morphism induced by $f$. Let $x^\prime\in X^\prime(\C)$ and $x:=f(x^\prime)\in X(\C)$. Then $f(Z^\R(x^\prime))\subseteq Z^{\R}(x)$.
\end{Lem}

\begin{Lem}\label{lemcomparezzr}
We have $Z^\R(x)\subseteq Z(x)$ and $Z^\R(x)$ is Zariski-dense in $Z(x)$. In particular, if $x$ is transcendental, then $\dim_\R Z^\R(x)\geq1$.
\end{Lem}
\begin{proof}
	It is clear that $Z^\R(x)\subseteq Z(x)$. After replacing $X_K$ by an irreducible affine open neighborhood of $\pi_K(x)$, we may assume that both $X_K$ and $X$ are affine. Let $h\in\sO(X)$ such that $h(Z^\R(x))=0$. Then we have
	$$h\otimes_\C 1\in\sO(X)\otimes_\C\tau(\sO(X))=\sC^{\R-\alg}(X)=\sC^{\R-\alg,L}(X)\otimes_K\C.$$
	View $\C$ as a vector space over $K$ and let $(e_j)_{j\in J}$ be a $K$-basis of $\C$. We may assume that $0\in J$ and $1=e_0$. Write $h\otimes_\C 1=\sum_{j\in J}g_je_j\in\sC^{\R-\alg,L}(X)\otimes_K\C$, where $g_j\in\sC^{\R-\alg,L}(X)$. Then $g_j(x)=0$ for every $j\in J$.
	
	Let $(f_n)_{n\in N}$ be a $K$-basis of $\sO^K(X)$. We may assume that $0\in N$ and $1=f_0$. Write
	$$g_j=\sum_{m,n\in N}b_{j,m,n}f_m\otimes_K\tau(f_n).$$
	Then we get
	$$h\otimes_\C 1=\sum_{j\in J, m,n\in N}b_{j,m,n}e_jf_m\otimes_K\tau(f_n).$$
	As $(e_jf_m\otimes\tau(f_n))_{j\in J, m,n\in N}$ forms a $K$-basis of
	$$\sC^{\R-\alg}(X)=\sO^K(X)\otimes_K\tau(\sO^K(X))\otimes_K\C,$$
	and $h\otimes_\C 1$ is contained in the image of $\sO(X)=\sO^K(X)\otimes_K\C$ in $\sC^{\R-\alg}(X)$, we have
	$$b_{j,m,n}=0$$
	for every $n\neq 0$. So $g_j=\sum_{m\in N}b_{j,m,0}f_m\otimes1\in\sO^K(X)$. Since $g_j(x)=0$, $g_j\mid_{Z(x)}=0$. Hence $h\mid_{Z(x)}=0$, which completes the proof.
\end{proof}

\begin{Lem}\label{lemvaluek}
	Assume that $X_K$ is affine. Let $h\in\sC^{\R-\alg,L}(X)$. For $x\in X(\C)$, if $h(x)\in K$, then $h$ is constant on $Z^\R(x)$.
\end{Lem}
\begin{proof}
	Write $h=g\circ\psi_\C$, where $g\in\sO^L(R(X))\otimes_L K$. Write $g=\pi_L^*(g_1)+\pi_L^*(g_2)i$, where $g_1,g_2\in\sO(R(X_K))$. Since $h(x)\in K$, both $\pi_L^*(g_1)(\psi_\C(x))$ and $\pi_L^*(g_2)(\psi_\C(x))$ are in $L$. The map $\pi_L\mid_{\psi_\C(x)}:\psi_\C(x)\to Y(x)_L$ induces an injection
	$$(\pi_L\mid_{\psi_\C(x)})^*:\sO(Y(x)_L)\hookrightarrow\R.$$
	We view elements of $L$ as constant functions in $\sO(Y(x)_L)$ and note that
	$$(\pi_L\mid_{\psi_\C(x)})^*(L)=L\subseteq\R.$$
	For $i\in\{1,2\}$, we have $g_i\mid_{Y(x)_L}\in\sO(Y(x)_L)$. Since
	$$(\pi_L\mid_{\psi_\C(x)})^*(g_i\mid_{Y(x)_L})=\pi_L^*(g_i)(\psi_\C(x))\in L,$$
	the injectivity of $(\pi_L\mid_{\psi_\C(x)})^*$ implies that $g_i\mid_{Y(x)_L}\in L$ for $i=1,2$. So both $\pi_L^*(g_1)$ and $\pi_L^*(g_2)$ are constant on $Y(x)=\pi_L^{-1}(Y(x)_L)\subseteq R(X)$, with values in $L$. We conclude that $h=g\circ\psi_{\C}=(\pi_L^*(g_1)+\pi_L^*(g_2)i)\circ\psi_{\C}$ is constant on $Z^\R(x)$, since $Z^\R(x)=\psi_{\C}^{-1}(Y(x)(\R))$.
\end{proof}

\subsection{Moduli space of rational maps}\label{sectionmoduli}
For $d\geq 2$, let $\Rat_d$ be the space of degree $d$ endomorphisms on $\P^1$. It is a smooth irreducible affine variety of dimension $2d+1$ \cite{Silverman2012}. The group $\PGL_2=\Aut(\P^1)$ acts on $\Rat_d$ by conjugacy. The geometric quotient $$\sM_d:=\Rat_d/\PGL_2=\Spec(\sO(\Rat_d)^{\PGL_2})$$ is the (coarse) \emph{moduli space} of endomorphisms of degree $d$ on $\P^1$, which is an irreducible affine variety of dimension $2d-2$, see \cite{Silverman2012,Silverman2007} for a detailed treatment. Let $$\Psi:\Rat_d\to\sM_d$$ be the quotient morphism. Let $\FL_d\subseteq\Rat_d$ be the locus of flexible Latt\`es maps, which is Zariski closed in $\Rat_d$. Set $[\FL_d]:=\Psi(\FL_d)$. The flexible-Latt\`es locus $[\FL_d]$ is non-empty if and only if $d$ is a square, and in this case $[\FL_d]$ is equidimensional of dimension $1$.

The above construction works over any algebraically closed field of characteristic $0$ and commutes with base changes.

\medskip

In the proof of Theorem \ref{thmalglen}, we need to consider only repelling periodic points of rational maps, and the notion of moduli spaces with certain marked periodic points is necessary. We introduce such moduli spaces here. See \cite[\S 8]{ji2023dao} for more details.

For every $n\in\Z_{>0}$, we construct the space $\Rat_d[n]$ of endomorphisms of degree $d$ on $\P^1$ with a single marked $n$-periodic point as follows. It is the closed subvariety of $\Rat_d\times\P^1$ defined by $f_{\Rat_d}^n(g,z)=(g,z)$, where
$$f_{\Rat_d}:\Rat_d\times\P^1\to\Rat_d\times\P^1,(g,z)\mapsto(g,g(z))$$
is the universal rational map. The first projection
$$\phi_n:\Rat_d[n]\to\Rat_d$$
is a finite morphism of degree $d^n+1$. Let
$$\la_n:\Rat_d[n]\to\A^1, (f_t, x)\mapsto df_t^n(x)\in\A^1$$
be the morphism given by the multiplier at the marked periodic point. 

Next, we introduce the moduli space of endomorphisms of degree $d$ with a finite sequence of marked periodic points whose periods satisfy certain conditions. Let $s_1,\dots, s_n\in\Z_{\geq 0}$ be finitely many integers such that $s_1\leq\dots\leq s_n$ and $s_i\leq d^{i!}+1$ for $1\leq i\leq n$. We construct the space ${\rm R_d}(s_1,\dots, s_n)$ of rational maps of degree $d$ with $s_n$ marked $n!$-periodic points, such that among these, $s_{n-1}$ are already $(n-1)!$-periodic, $\dots$, and among these $s_2$ marked $2!$-periodic points, $s_1$ are already $1$-periodic (i.e., fixed), always counted with multiplicity, as follows. An $m$-periodic point of $f$ means a fixed point of $f^m$, and we do not require its exact period to be $m$; hence, an $m!$-periodic point is automatically $k!$-periodic for $k\geq m$. Consider the fiber product $(\Rat_d[n!])^{s_n}_{/\Rat_d}$ of $s_n$ copies of $\Rat_d[n!]$ over $\Rat_d$. For $i\neq j\in\{1,\dots, s_n\}$, let
$$\pi_{i,j}: (\Rat_d[n!])^{s_n}_{/\Rat_d}\to (\Rat_d[n!])^2_{/\Rat_d}$$
be the projection to the $i$-th and $j$-th coordinates. The diagonal $\Delta\subseteq (\Rat_d[n!])^2_{/\Rat_d}$ is an irreducible component of $(\Rat_d[n!])^2_{/\Rat_d}$. Consider the open subset
$$U:=(\Rat_d[n!])^{s_n}_{/\Rat_d}\setminus (\cup_{i\neq j\in\{1,\dots,s_n\}}\pi_{i,j}^{-1}(\Delta))$$
in which the $s_n$ marked points are pairwise distinct. Let $U^\prime$ be the subset of $U$ consisting of points $(f, x_1,\dots, x_{s_n})$ satisfying $f^{m!}(x_i)=x_i$ for all integers $1\leq m\leq n$ and $1\leq i\leq s_m$. This means that in $U^\prime$, the first $s_m$ points in the whole sequence of $s_n$ marked $n!$-periodic points are indeed $m!$-periodic, for all integers $1\leq m\leq n$. The set $U^\prime$ is open and closed in $U$. We then define ${\rm R_d}(s_1,\dots,s_n)$ to be the Zariski closure of $U^\prime$ in $(\Rat_d[n!])^{s_n}_{/\Rat_d}$. For all integers $0\leq m\leq n$, define
$$\phi_{n,m}: {\rm R_d}(s_1,\dots, s_n)\to {\rm R_d}(s_1,\dots, s_m):(f, x_1,\dots, x_{s_n})\mapsto (f, x_1,\dots, x_{s_m})$$
to be the restriction morphism, where $s_0:=0$ and when $m=0$ the morphism is given by
$$\phi_{n,0}: {\rm R_d}(s_1,\dots, s_n)\to\Rat_d:(f, x_1,\dots, x_{s_n})\mapsto f.$$
For all integers $0\leq m_1\leq m_2\leq n$, we have
$$\phi_{m_2,m_1}\circ\phi_{n,m_2}=\phi_{n,m_1}.$$
We define the multiplier morphism on ${\rm R_d}(s_1,\dots, s_n)$ to be the morphism
$$\la_{s_1,\dots,s_n}: {\rm R_d}(s_1,\dots, s_n)\to\A^{s_n}:(f, x_1,\dots,x_{s_n})\mapsto (df^{n!}(x_1),\dots, df^{n!}(x_{s_n})).$$

Recall that for a rational map $g(z)\in\C(z)$ of degree $\geq2$ and a fixed point $z_0\in\P^1(\C)$ of $g$, the multiplicity of the fixed point $z_0$ of $g$ is the order of $g(z)-z=0$ at $z_0$ when $z_0\neq\infty$, and we can define the multiplicity using a conjugation to move $z_0$ to $\C$ when $z_0=\infty$. For every $k\geq1$, it is clear that a fixed point $x$ of $f_t^k$ has multiplicity $1$ if and only if
\begin{equation}\label{multifixed}
	\la_k(f_t,x)=df_t^k(x)=\rho_{f^k}(x)\neq1.
\end{equation}
Then the projection $\phi_k:\Rat_d[k]\to\Rat_d$ is \'etale at every point $(f_t,x)\in\Rat_d[k]\setminus\la_k^{-1}(1)$. Thus we conclude that $\phi_{n,0}:{\rm R_d}(s_1,\dots,s_n)\to\Rat_d$ is \'etale at every point in $ (\la_{s_1,\dots,s_n})^{-1}((\A^1\setminus\{1\})^{s_n})$. 

Define
$$\sM_d(s_1,\dots,s_n):={\rm R_d}(s_1,\dots, s_n)/\PGL_2$$
to be the moduli space of endomorphisms of degree $d$ on $\P^1$ with $s_n$ marked $n!$-periodic points, such that among these, $s_{n-1}$ are already $(n-1)!$-periodic, $\dots$, and among these $s_2$ marked $2!$-periodic points, $s_1$ are already $1$-periodic (i.e., fixed), always counted with multiplicity. The moduli space $\sM_d(s_1,\dots,s_n)$ is affine, which makes the construction relatively easy. The morphisms $\phi_{n,m}$ and $\la_{s_1,\dots, s_n}$ descend to
$$[\phi_{n,m}]:\sM_d(s_1,\dots,s_n)\to\sM_d(s_1,\dots,s_m)$$
when $1\leq m\leq n$,
$$[\phi_{n,0}]:\sM_d(s_1,\dots,s_n)\to\sM_d,$$
and
$$[\la_{s_1,\dots, s_n}]:\sM_d(s_1,\dots, s_n)\to\A^{s_n}.$$
Then $[\phi_{n,0}]$ is \'etale at every point $x\in [\la_{s_1,\dots,s_n}]^{-1}((\A^1\setminus\{1\})^{s_n})$.

All the schemes $\Rat_d$, $\sM_d$, ${\rm R_d}(s_1,\dots, s_n)$, and $\sM_d(s_1,\dots,s_n)$ considered here are defined over $\Q$.

\subsection{Length maps}\label{sectionlengmap}
Recall that $K$ is an algebraically closed subfield of $\C$ which is invariant under complex conjugation. We use the terminology and notations from \S\ref{sectionweilrest}--\ref{sectionmoduli}.

For $d\geq 2$, let $s_1,\dots, s_n\in\Z_{\geq 0}$ be finitely many integers such that $s_1\leq\dots\leq s_n$ and $s_i\leq d^{i!}+1$ for $1\leq i\leq n$. Let
$$|\la_{s_1,\dots,s_n}|:\sM_d(s_1,\dots, s_n)(\C)\to\R_{\geq 0}^{s_n}$$
be the composition of
$$[\la_{s_1,\dots, s_n}]:\sM_d(s_1,\dots, s_n)(\C)\to\C^{s_n}$$
and the absolute value map
$$(a_1,\dots, a_{s_n})\in\C^{s_n}\mapsto (|a_1|,\dots, |a_{s_n}|)\in\R_{\geq 0}^{s_n}.$$
Define $$q_{s_1,\dots, s_n}:\sM_d(s_1,\dots, s_n)(\C)\to\R_{\geq 0}^{s_n}$$ to be the composition of
$$|\la_{s_1,\dots,s_n}|:\sM_d(s_1,\dots, s_n)(\C)\to\R_{\geq 0}^{s_n}$$
and the map
$$(a_1,\dots, a_{s_n})\in\R_{\geq 0}^{s_n}\mapsto(a_1^2,\dots, a_{s_n}^2)\in\R_{\geq 0}^{s_n}.$$
It is clear that
$$q_{s_1,\dots, s_n}\in\sC^{\R-\alg,L}(\sM_d(s_1,\dots, s_n)_\C).$$
Here the model of
$$\sM_d(s_1,\dots, s_n)_{\C}:=\sM_d(s_1,\dots, s_n)\times_{\Spec(\Q)}\Spec(\C)$$
over $K$ is taken to be
$$\sM_d(s_1,\dots, s_n)_K:=\sM_d(s_1,\dots, s_n)\times_{\Spec(\Q)}\Spec(K).$$

\medskip

By Lemma \ref{lemvaluek}, for every $x\in\sM_d(s_1,\dots, s_n)(\C)$, if $q_{s_1,\dots, s_n}(x)\in L^{s_n}$, then $q_{s_1,\dots, s_n}|_{Z^{\R}(x)}$ is constant. Hence for every $x\in\sM_d(s_1,\dots, s_n)(\C)$, if $|\la_{s_1,\dots,s_n}|(x)\in L^{s_n}$, then $|\la_{s_1,\dots,s_n}||_{Z^{\R}(x)}$ is constant. 

\subsection{Rigidity of length spectrum}\label{sectionlenrig}
In this subsection, we recall the rigidity of length spectrum proved by Ji-Xie \cite{Ji2023}.

Let $f$ be an endomorphism of $\P^1(\C)$ of degree $d\geq 2$. As in \cite[\S 8.3]{Ji2023}, the \emph{length spectrum}
$$L(f)=\{L(f)_n, n\geq 1\}$$
of $f$ is a sequence of finite multisets\footnote{A multiset is a set except allowing multiple instances for each of its elements. The number of instances of an element is called its multiplicity. For example: $\{a,a,b,c,c,c\}$ is a multiset of cardinality $6$, with multiplicities $2,1,3$ for $a,b,c$, respectively.}, where $L(f)_n:=L_n(f)$ is the multiset of lengths of all fixed points of $f^n$. In particular, $L(f)_n$ is a multiset of non-negative real numbers of cardinality $d^n+1$. For every $n\geq 1$, let $RL(f)_n$ be the sub-multiset of $L(f)_n$ consisting of all elements $>1$. We call
$$RL(f):=\{RL(f)_n, n\geq 1\}$$
the \emph{repelling length spectrum} of $f$ and its subsequence
$$RL^*(f):=\{RL^*(f)_n:=RL(f)_{n!}, n\geq 1\}$$
over the factorials $n!$ the \emph{main repelling length spectrum} of $f$. By definition, we have $d^n+1\geq\# RL(f)_n$ for every $n\geq 1$. Observe that the sequence of differences
$$(d^{n!}+1-\# RL^*(f)_n)_{n\geq1}$$
is non-decreasing, since every $n!$-periodic point is $(n+1)!$-periodic and the multiplicity of a repelling fixed point of $f^{n!}$ is $1$ by \eqref{multifixed}. Also, the sequence $(d^{n!}+1-\# RL^*(f)_n)_{n\geq1}$ is bounded from above, since every rational map on $\P^1(\C)$ of degree $d\geq2$ has only finitely many periodic cycles whose multiplier has absolute value $\leq1$ (see \cite[Theorem 13.1]{Milnor}). Since $L(f), RL(f)$, and $RL^*(f)$ are invariant under conjugacy, they descend to $\sM_d(\C)$. For every $[f]\in\sM_d(\C)$, define
$$L([f]):=L(f), RL([f]):=RL(f),RL^*([f]):=RL^*(f)$$
for any $f$ in the class $[f]$.

\medskip

Let $\Om$ be the set of sequences $\{A_n, n\geq 1\}$ of multisets consisting of real numbers of absolute values $>1$ satisfying $\#A_n\leq d^{n!}+1$ for $n\geq1$, and such that for every $a\in A_n$ with multiplicity $m\geq1$, we have $a^{n+1}\in A_{n+1}$ with multiplicity $\geq m$. For $A,B\in\Om$, we write $A\subseteq B$ if $A_n\subseteq B_n$ as multisets for every $n\geq 1$. An element $A=\{A_n\}\in\Om$ is called \emph{big} if $(d^{n!}+1-\#A_n)_{n=1}^\infty$ is bounded. For every endomorphism $f$ of $\P^1(\C)$ of degree $d$, we have $RL^*(f)\in\Om$ and it is big.

For a big $A\in\Om$, Ji-Xie proved the following rigidity result on length spectrum, which is a main ingredient in the proof of Theorem \ref{thmalglen}.

\begin{Thm}\cite[Theorem 8.25]{Ji2023}\label{thmbigspecrig}
If $A\in\Om$ is big, then the set
$$\{[f]\in\sM_d(\C)\setminus[\FL_d]: A\subseteq RL^*(f)\}$$
is finite.
\end{Thm}

\subsection{Proof of Theorem \ref{thmalglen}}\label{sectionalglenpf}
Let $f:\P_{\C}^1\to\P_{\C}^1$ be a rational map of degree $d\geq 2$. Assume that $f$ is not a flexible Latt\`es map and that for every $x\in\Per(f)(\C)$, $|\rho_f(x)|\in K$. We want to show that $[f]\in\sM_d(\C)$ is not transcendental over $K$ for the model $(\sM_d)_K$. Now assume, for contradiction, that $[f]$ is transcendental.

Set $A:=RL^*(f)\in\Om$, which is big. Set $s_n:=\#A_n$ for $n\geq1$. We may pick a sequence of periodic points $(x_i)_{i=1}^\infty$ such that for every $n\geq 1$, the first $s_n$ points $x_1,\dots, x_{s_n}$ are fixed by $f^{n!}$, and we have
$$A_n=\left\lbrace\left|\rho_{f^{n!}}(x_i)\right|: i=1,\dots, s_n\right\rbrace.$$
Let $[f_n]\in\sM_d(s_1,\dots,s_n)(\C)$ be the point represented by $(f,x_1,\dots,x_{s_n})$. It is clear that $[\phi_{n,0}]([f_n])=[f]$ for every $n\geq 1$.

Since $[f]$ is transcendental, for every $n\geq 1$, $[f_n]$ is also transcendental. By Lemma \ref{lemcomparezzr}, we get $\dim_{\R}(Z^{\R}([f_n]))\geq 1$ for every $n\geq 1$. Our assumption implies that
$$|\la_{s_1,\dots, s_n}|([f_n])\in L^{s_n},n\geq1.$$
Then from the last paragraph of \S \ref{sectionlengmap}, we see that $|\la_{s_1,\dots, s_n}|$ is constant on $Z^{\R}(f_n)$. As $|\la_{s_1,\dots, s_n}|([f_n])\in (1,+\infty)^{s_n}$, the morphism $[\phi_{n,0}]$ is \'etale in a neighborhood of $Z^{\R}([f_n])$. Since $[\phi_{n,0}]$ is a finite morphism, the subset
$$V_n:=[\phi_{n,0}](Z^{\R}([f_n]))$$
is closed in $\sM_d(\C)$ and it is an admissible subset of $\sM_d(\C)$. Moreover, by Lemma \ref{lemmorphismzx}, the sequence $(V_n)_{n\geq 1}$ is decreasing. By Theorem \ref{thmNoetherianad}, there is $N\geq 1$ such that $V_n=V_N$ for all $n\geq N$.

Then for every $g\in V_N$, we have $A\subseteq RL^*([g])$. Since $[f]\not\in [\FL_d]$, the set $Z^{\R}([f_N])$ is real irreducible with $\dim_{\R}Z^{\R}([f_N])\geq 1$, so $V_N\cap (\sM_d(\C)\setminus [\FL_d])$ is infinite. This contradicts Theorem \ref{thmbigspecrig}, which completes the proof.
\qed

\section{An equidistribution theorem}\label{sectionequid}
Equidistribution-type results have played a significant role in arithmetic geometry and arithmetic dynamics over the past decades. In this section, we state and prove a reformulation (Theorem \ref{equid}) of Yuan's equidistribution theorem \cite[Theorem 3.1]{Yuan2008}. We present this theorem because existing equidistribution theorems in the literature (e.g., \cite{Yuan2008}) typically handle the Galois orbit of a single small point at a time, which is insufficient for our proof of Theorem \ref{thmmainnumf}. Our statement is slightly stronger than \cite[Theorem 3.1]{Yuan2008} as we work with sequences of (weighted) countable Galois-invariant subsets, which is particularly useful for applications in dynamics. The proof of Theorem \ref{equid} is based on applying Yuan's original equidistribution theorem \cite[Theorem 3.1]{Yuan2008}. We follow the terminology in \cite{Yuan2008}.

For the convenience of readers, we briefly review some basic notions in Arakelov geometry. More details can be found in \cite{Zhang1995,Yuan2008,MoriwakiAra,YZ21}.

Let $K$ be a number field, $X$ a projective variety of dimension $n-1$ over $K$, and $\sL$ a line bundle on $X$. We fix an embedding $\overline{K}\to\C_v$ for each place $v\in\sM_K$ of $K$. An \emph{adelic metric} on $\sL$ is a $\C_v$-norm $\|\cdot\|_v$ on the fiber $\sL_{\C_v}(x)$ for every $x\in X(\overline{K})$ and $v\in\sM_K$, satisfying certain continuity and coherence conditions; see \cite[\S 1]{Zhang1995} for a precise definition.

An adelic metric on $\sL$ is called an \emph{algebraic metric} (or a \emph{model metric}) if it is induced by an $\sO_K$-model $(\sX,\tilde{\sL})$ of $(X,\sL^e)$ for some $e\in\Z_{>0}$ (i.e., $\sX$ is an integral scheme projective and flat over $\sO_K$, and $\tilde{\sL}$ is a Hermitian line bundle over $\sX$ such that the generic fiber of $(\sX,\tilde{\sL})$ is $(X,\sL^e)$). Here, the induced metric at an archimedean place is induced by the Hermitian metric of $\tilde{\sL}$, while for a non-archimedean place $v$ and $x\in X(\overline{K})$, the induced metric on $\sL_{\C_v}(x)$ is determined by the lattice $(\tilde{x}^*\tilde{\sL}_{\sO_{\C_v}})^{1/e}$ in $\sL_{\C_v}(x)$, where $\tilde{x}:\Spec(\sO_{\C_v})\to\sX_{\sO_{\C_v}}$ is the extension of $x$. The metric induced by such a model always satisfies the conditions of an
adelic metric; see \cite[\S 1]{Zhang1995}. The algebraic metric is called \emph{semipositive} if $\tilde{\sL}$ has semipositive curvature at all archimedean places and non-negative degree on every vertical curve of $\sX$.

An (adelic) metric on $\sL$ is called \emph{semipositive} if it is the uniform limit of a sequence of semipositive algebraic metrics on $\sL$. Assume that $\sL$ is ample and equipped with a semipositive (adelic) metric (which always exists). The resulting metrized line bundle is denoted by $\overline{\sL}$. Then for every closed subvariety $Y$ of $X_{\overline{K}}$, we can define its $\overline{\sL}$-height $h_{\overline{\sL}}(Y)$ by $$h_{\overline{\sL}}(Y)=\frac{\hat{c}_1(\overline{\sL}|_{\overline{Y}})^{\dim(Y)+1}}{(\dim(Y)+1)\deg_{\sL}(\overline{Y})},$$ where $\overline{Y}$ is the image of $Y\to X_{\overline{K}}\to X$ (i.e., the closed subvariety of $X$ corresponding to the Galois orbit of $Y$), and $\hat{c}_1(\overline{\sL}|_{\overline{Y}})^{\dim(Y)+1}$ is the arithmetic self-intersection number of $\overline{\sL}|_{\overline{Y}}$ on $\overline{Y}$. In particular, for every algebraic point $x\in X(\overline{K})$, we get its height $h_{\overline{\sL}}(x)$, and the function $h_{\overline{\sL}}:{X(\overline{K})}\to\R$ is a Weil height associated to $\sL$ (up to a normalizing factor $[K:\Q]$). For an algebraic point $x\in X(\overline{K})$, we have $$h_{\overline{\sL}}(x)=\frac{1}{\deg(x)}\sum_{v\in\sM_K}\sum_{z\in O(x)}(-\log\|s(z)\|_v),$$
where $O(x)$ is the Galois orbit of $x$ and $s\in\Gamma(X,\sL)$ is a section non-vanishing on $O(x)$.

\begin{Thm}\label{equid}
	Let $K\subset\overline{\Q}$ be a number field and $X$ be a projective variety over $K$. Fix an embedding of $\overline{\Q}$ into $\C$. Let $\overline{\sL}$ be a metrized line bundle on $X$ such that $\sL$ is ample and the metric is semipositive. Let $\mu:=\deg_{\sL}(X)^{-1}c_1(\overline{\sL})^{\dim(X)}_{\C}$ be the canonical probability measure on $X(\C)$ associated to $\overline{\sL}$. For $n\in\Z_{>0}$, let $S_n$ be a countable subset of $X(\overline{\Q})$ which is $\Gal(\overline{\Q}/K)$-invariant. For every $n\geq1$ and $y\in S_n$, fix a real number $a_{n,y}>0$ such that $\sum_{y\in S_n}a_{n,y}=1$ and $a_{n,y}=a_{n,\sigma(y)}$ for all $n\geq1$, $y\in S_n$, and $\sigma\in\Gal(\overline{\Q}/K)$. Assume that the following two conditions hold:
	
	(1) (small) $\sum_{y\in S_n}a_{n,y}h_{\overline{\sL}}(y)\to h_{\overline{\sL}}(X)$ as $n\to +\infty$;
	 
	(2) (generic) for every proper subvariety $V\subsetneqq X$, $\sum_{y\in S_n\cap V}a_{n,y}\to0$ as $n\to +\infty$.
	
	Then the measure $\mu_n:=\sum_{y\in S_n}a_{n,y}\delta_y$ converges weakly to $\mu$ on $X(\C)$ as $n\to+\infty$, where $\delta_y$ denotes the Dirac measure at the point $y$; i.e., for every continuous function $g$ on $X(\C)$ (with the complex topology), we have
	\begin{equation}\label{2.3}
		\lim\limits_{n\to\infty}\sum_{y\in S_n}a_{n,y}g(y)=\int_{X(\C)}g d\mu.
	\end{equation}
\end{Thm}

\begin{proof}
For every $n\geq 1$, write $S_n=\sqcup_{i=1}^{m_n} O_{n,i}$, where $O_{n,i}$ are the distinct Galois orbits contained in $S_n$ and $m_n\in\Z_{>0}\cup\{\infty\}$. 

In \cite[Theorem 3.1]{Yuan2008}, Yuan proved the theorem for the case where each $S_n$ ($n\geq1$) consists of a single Galois orbit, i.e., $m_n=1$ for every $n\geq1$.

\medskip

By contradiction, after passing to a subsequence of $(S_n)$, we may assume that there is a continuous (real-valued) function $g$ on $X(\C)$ and a constant $c\in\R_{>0}$ such that $\int_{X(\C)} g d\mu=0$ and 
\begin{equation}\label{equcongec}
	\int_{X(\C)}g d\mu_n\geq c>0
\end{equation}
for every $n\geq 1$.

As $a_{n,y}$ is constant on Galois orbits, we can define $a_{n,i}:=a_{n,y}$ for an arbitrary $y\in O_{n,i}$, for all $n\geq1$ and $1\leq i\leq m_n$. For all $n\geq1$ and $1\leq i\leq m_n$, set $\mu_{n,i}:=\sum_{y\in O_{n,i}}\delta_y$. Then we have 
$$\sum_{i=1}^{m_n}a_{n,i}(\#O_{n,i})=1\quad\text{and}\quad\mu_n=\sum_{i=1}^{m_n}a_{n,i}\mu_{n,i}$$
for every $n\geq1$. As $h_{\overline{\sL}}(y)$ is constant on Galois orbits, we can define $h_{n,i}:=h_{\overline{\sL}}(y)$ for an arbitrary $y\in O_{n,i}$, for all $n\geq1$ and $1\leq i\leq m_n$.

The small condition (1) implies that $\sum_{i=1}^{m_n}a_{n,i}(\#O_{n,i})h_{n,i}\to h_{\overline{\sL}}(X)$ as $n\to+\infty$. After taking a suitable subsequence, we may assume that
\begin{equation}\label{equationsmall}
	\sum_{i=1}^{m_n}a_{n,i}(\#O_{n,i})h_{n,i}\leq h_{\overline{\sL}}(X)+\frac{1}{n^2}
\end{equation} for every $n\geq 1$.

As $X$ has only countably many subvarieties defined over $\overline{\Q}$, there is an increasing sequence $V_1\subseteq V_2\subseteq\cdots\subseteq V_n\subseteq\cdots\subsetneqq X$ of proper subvarieties of $X$ defined over $K$ such that for every proper subvariety $V$ of $X$ defined over $\overline{\Q}$, there exists $N=N(V)\in\Z_{>0}$ with $V\subseteq V_n$ for all integers $n\geq N$. (By considering the Galois orbits of subvarieties, we may assume that the $V_n$ are defined over $K$.) By Zhang's theorem of successive minima \cite[Theorem 1.10]{Zhang1995}, we may further assume that
\begin{equation}\label{houtVn}
	h_{\overline{\sL}}(y)\geq h_{\overline{\sL}}(X)-\frac{1}{n^2}
\end{equation}
for every $n\geq1$ and $y\in X(\overline{\Q})\setminus V_n$. Set
$$G_n:=\{i\in[1,m_n]\cap\Z: O_{n,i}\cap V_n=\emptyset\}$$
and
$$B_n:=([1,m_n]\cap\Z)\setminus G_n=\{i\in[1,m_n]\cap\Z: O_{n,i}\subseteq V_n\},$$
for $n\geq1$. (Note that for every subvariety $V$ of $X$ defined over $K$, either $O_{n,i}\cap V=\emptyset$ or $O_{n,i}\subseteq V$.)

The generic condition (2) implies that after passing to a subsequence of $(S_n)$, we may assume that
$$1-\frac{1}{n^2}\leq\sum_{i\in G_n}a_{n,i}(\#O_{n,i})\leq 1,$$
for every $n\geq1$. Equivalently,
\begin{equation}\label{sumBn}
	0\leq\sum_{i\in B_n}a_{n,i}(\#O_{n,i})\leq\frac{1}{n^2},
\end{equation} for every $n\geq1$.

Since $\sL$ is ample, there is a constant $M\in\R_{\geq0}$ such that
\begin{equation}\label{Mbd}
	h_{\overline{\sL}}(y)\geq h_{\overline{\sL}}(X)-M
\end{equation}
for every $y\in X(\overline{\Q})$.

By \eqref{equationsmall}, \eqref{Mbd}, and \eqref{sumBn}, we get
\begin{align*}
	\sum_{i\in G_n}a_{n,i}(\#O_{n,i})h_{n,i}&\leq h_{\overline{\sL}}(X)+\frac{1}{n^2}-\sum_{i\in B_n}a_{n,i}(\#O_{n,i})h_{n,i}\\
	&\leq h_{\overline{\sL}}(X)+\frac{1}{n^2}-\sum_{i\in B_n}a_{n,i}(\#O_{n,i})(h_{\overline{\sL}}(X)-M)\\
	&= (\sum_{i\in G_n}a_{n,i}(\#O_{n,i}))h_{\overline{\sL}}(X)+\frac{1+Mn^2\sum_{i\in B_n}a_{n,i}(\#O_{n,i})}{n^2}\\
	&\leq (\sum_{i\in G_n}a_{n,i}(\#O_{n,i}))h_{\overline{\sL}}(X)+\frac{1+M}{n^2}.
\end{align*}
Set $H_n:=\{i\in G_n:h_{n,i}\leq h_{\overline{\sL}}(X)+1/n\}$ and $I_n:=G_n\setminus H_n$ for $n\geq1$. By \eqref{houtVn}, we have
\begin{align*}
	\frac{1+M}{n^2}&\geq\sum_{i\in G_n}a_{n,i}(\#O_{n,i})(h_{n,i}-h_{\overline{\sL}}(X))\geq-\frac{\sum_{i\in H_n}a_{n,i}(\#O_{n,i})}{n^2}+\frac{\sum_{i\in I_n}a_{n,i}(\#O_{n,i})}{n}\\
	&\geq -\frac{1}{n^2}+\frac{\sum_{i\in I_n}a_{n,i}(\#O_{n,i})}{n}.
\end{align*} 
Therefore, 
$$\sum_{i\in I_n}a_{n,i}(\#O_{n,i})\leq\frac{2+M}{n}$$
for every $n\geq1$. It follows that $$\sum_{i\in H_n}a_{n,i}(\#O_{n,i})\geq1-\frac{1}{n^2}-\frac{2+M}{n}\to 1$$ as $n\to+\infty$.

Set $\mu_n^\prime:=\sum_{i\in H_n}a_{n,i}\mu_{n,i}$ for $n\geq1$. Then $\mu_n-\mu_n^\prime$ is a (positive) measure on $X(\C)$ whose total mass tends to $0$ as $n\to+\infty$. After removing finitely many terms, we may assume that $\int_{X(\C)}g d\mu_n^\prime\geq c/2$ for every $n\geq 1$, since $g$ is bounded on the compact set $X(\C)$. Hence there is $i_n\in H_n$ such that $\int_{X(\C)}g d\mu_{n,i_n}\geq c/2$; fix such an $i_n$ for each $n\geq1$. Since $O_{n,i_n}\cap V_n=\emptyset$, the sequence $(O_{n,i_n})_n$ of Galois orbits satisfies the generic condition (2) by construction of $(V_n)$. By \eqref{houtVn} and the definition of $H_n$, we have
$$h_{\overline{\sL}}(X)-\frac{1}{n^2}\leq h_{n,i_n}\leq h_{\overline{\sL}}(X)+\frac{1}{n}$$
for every $n\geq 1$, i.e., the small condition (1) holds for $(O_{n,i_n})_n$. Applying Yuan's equidistribution theorem \cite[Theorem 3.1]{Yuan2008} to $(O_{n,i_n})_n$, we get $\mu_{n,i_n}\to\mu$ as $n\to+\infty$. Thus,
$$0<\frac{c}{2}\leq\lim_{n\to\infty}\int_{X(\C)}g d\mu_{n,i_n}=\int_{X(\C)}g d\mu=0,$$
which is a contradiction.
\end{proof}

\begin{Rem}
	For a place $v\in\sM_K$, we can consider two canonical analytic $v$-spaces of $X$: the $\C_v$-analytic space $X^{an}_{\C_v}$ and the $K_v$-analytic space $X^{an}_{K_v}$. We refer to the former as the geometric case and the latter as the algebraic case; see \cite[\S 3.1]{Yuan2008}. Theorem \ref{equid} deals with the geometric case at an archimedean place. The same idea applies to non-archimedean places or the algebraic case, which gives the full analogy of \cite[Theorems 3.1 and 3.2]{Yuan2008}.
\end{Rem}

To verify the ``generic'' condition in Theorem \ref{equid}, we need the following lemma. The proof uses ergodic theory with respect to the constructible topology (on algebraic varieties) introduced by Xie \cite{Xie2023}. 
\begin{Lem}\label{pergene}
	Let $K$ be a number field and $X$ be a projective variety over $K$. Given a dominant endomorphism $f:X\to X$ and a sequence $(x_n)_{n=1}^\infty$ of pairwise distinct $f$-periodic points in $X(\overline{K})$. Assume that $(x_n)_{n=1}^\infty$ is generic in $X$, i.e., every proper Zariski closed subset $Z\subsetneqq X$ contains only finitely many $x_n$. Then for every proper subvariety $V\subsetneqq X$, we have
	\begin{equation}\label{2.4}
		\frac{\#(V\cap O_f(x_n))}{\# O_f(x_n)}\to0,\text{ as }n\to+\infty,
	\end{equation}where $O_f(x_n)$ is the (forward) orbit of $x_n$ under $f$.
\end{Lem}
\begin{proof}
	Clearly, it suffices to show that for every subsequence $(n_k)_k$ of $(n)_{n=1}^\infty$, there exists a subsubsequence $(n_{k_l})_l$ such that
	$$\frac{\#(V\cap O_f(x_{n_{k_l}}))}{\# O_f(x_{n_{k_l}})}\to0,\text{ as }l\to+\infty.$$ 
	
	Fix a proper subvariety $V$ in $X$. Let $\lvert X\rvert$ denote $X$ equipped with the constructible topology (i.e., the topology generated by all Zariski closed and open subsets of $X$) and let $\sM^1(\lvert X\rvert)$ be the space of all probability Radon measures on $\lvert X\rvert$ with the topology of weak convergence relative to all continuous functions on $\lvert X\rvert$. Then $\sM^1(\lvert X\rvert)$ is sequentially compact (cf. \cite[Corollary 1.14]{Xie2023}). For $n\in\Z_{>0}$, set
	$$m_n=(\# O_f(x_n))^{-1}\sum_{z\in O_f(x_n)}\delta_z.$$
	Let $1_V$ be the characteristic function of $V\subsetneqq\lvert X\rvert$; then $1_V$ is continuous with respect to the constructible topology. By the sequential compactness of $\sM^1(\lvert X\rvert)$, it suffices to show that for every subsequence $(n_k)_k$ of $(n)_{n=1}^\infty$ with $m_{n_k}\to m$ as $k\to+\infty$ in $\sM^1(\lvert X\rvert)$ for some $m\in\sM^1(\lvert X\rvert)$, we have
	\begin{equation}\label{equintonev}
		\int 1_V dm=0.
	\end{equation}
	
	Without loss of generality, assume that $(m_n)$ itself converges to a measure $m\in\sM^1(\lvert X\rvert)$; it suffices to prove \eqref{equintonev} in this case. As $f_*m_n=m_n$ for all $n\geq1$, we have $f_*m=m$ as well. Then according to \cite[Lemma 5.3]{Xie2023}, $m$ must be of the form $m=\sum_{y\in S}a_y\delta_{O_f(y)}$, where $S$ is a countable set of $f$-periodic points in $\lvert X\rvert$ (note that the points in $S$ may not be closed with respect to the Zariski topology of $X$), $a_y\in\R_{>0}$ with $\sum_{y\in S}a_y=1$, and $\delta_{O_f(y)}=(\#O_f(y))^{-1}\sum_{z\in O_f(y)}\delta_z$ for $y\in S$. As $m_n\to m$, we get
	$$\frac{\#(V\cap O_f(x_n))}{\# O_f(x_n)}=\int 1_V dm_n\to\int 1_V dm,\text{ as }n\to+\infty.$$
	Suppose that \eqref{equintonev} fails. Then there must be a point $y\in S$ (in the scheme-theoretic sense) with $a_y>0$ and $V\cap O_f(y)\neq\emptyset$. Let $k$ be the exact period of $y$ under $f$, and let $Y$ be the Zariski closure of $\{y\}$. Then $Y\subseteq\cup_{j=0}^{k-1}f^j(V)$, so $Y$ is a proper Zariski closed subset of $X$. Note that $$\frac{\#(Y\cap O_f(x_n))}{\# O_f(x_n)}=\int 1_Y dm_n\to\int 1_Y dm\geq\frac{a_y}{k}>0,\text{ as }n\to+\infty.$$ Hence for all sufficiently large $n\gg1$, we have
	\begin{equation}\label{infxn}
		x_n\in\cup_{j=0}^\infty f^j(Y)=\cup_{j=0}^{k-1} f^j(Y),
	\end{equation}
	where the last equality follows from $f^k(y)=y$ and $Y=\overline{\{y\}}^{\Zar}$. Since
	$$\dim(\cup_{j=0}^{k-1} f^j(Y))=\dim(Y)<\dim(X),$$
	the set $\cup_{j=0}^{k-1} f^j(Y)$ contains only finitely many $x_n$ because $(x_n)$ is generic, contradicting \eqref{infxn}.
\end{proof}

\section{Proofs of Theorem \ref{thmmainnumf} and the defined over $\overline{\Q}$ case of Theorem \ref{thmmaindim}}\label{section4}
\begin{proof}[Proof of Theorem \ref{thmmainnumf}]
Recall that $f:\P^1(\C)\to\P^1(\C)$ is a rational map of degree $d\geq2$, and there is a number field $K$ such that for every $z_0\in\Per(f)$, we have $\lvert\rho_f(z_0)\rvert^n\in K$ for some integer $n=n(z_0)\geq1$. We need to show that such an $f$ must be exceptional (Definition \ref{excep}).

Assume by contradiction that $f$ is not exceptional. In particular, $f$ is not a flexible Latt\`es map. For every $z_0\in\Per(f)$, we have $\lvert\rho_f(z_0)\rvert\in\overline{\Q}$, so $f$ is defined over $\overline{\Q}$ (hence over some number field) by Theorem \ref{thmalglenintro}. After replacing $K$ by a suitable finite extension, we may assume that both $f$ and $\overline{f}$ are defined over $K$, where $\overline{f}$ denotes the rational map obtained by conjugating the coefficients of $f$. According to \cite[Theorem 7 and Lemma 9]{Huguin2023}, which are essentially due to Zdunik \cite{zdunik2014characteristic}, there exists a sequence $(x_n)_{n=1}^\infty$ of distinct points in $\Per^*(f)$ such that $$a:=\lim\limits_{n\to\infty}\chi_f(x_n)>\sL_f,$$ where the limit exists and is finite.

Clearly, $\sL_f=\sL_{\overline{f}}$. For any $x\in\Per(f)$, we have
$$\overline{x}\in\Per(\overline{f}),\quad n_f(x)=n_{\overline{f}}(\overline{x}),\quad\text{and}\quad\rho_f(x)=\overline{\rho_{\overline{f}}(\overline{x})},$$
so $\chi_f(x)=\chi_{\overline{f}}(\overline{x})$.

Consider the morphism
$$F:=f\times\overline{f}:\P^1\times\P^1\to\P^1\times\P^1$$
over $K$. For $n\in\Z_{>0}$, set $p_n=(x_n,\overline{x_n})\in\Per(F)$. Let $\Gamma$ be the Zariski closure of $\{p_n\colon n\in\Z_{>0}\}$ in $\P^1\times\P^1$, which is defined over $\overline{\Q}$ since the $p_n$ are defined over $\overline{\Q}$. As the points $p_n$ are pairwise distinct, we have $\dim(\Gamma)\geq1$. After passing to a subsequence of $(p_n)$, we may assume that $\Gamma$ is irreducible and that $(p_n)_{n=1}^\infty$ is generic in $\Gamma$, i.e., no proper subvariety of $\Gamma$ contains infinitely many $p_n$. Let $\pi_j:\Gamma\to\P^1$ be the $j$-th projection from $\Gamma\subseteq\P^1\times\P^1$ to the $j$-th factor, for $j=1,2$. It is clear that $\pi_1\circ F=f\circ\pi_1$ and $\pi_2\circ F=\overline{f}\circ\pi_2$. 

There are 2 cases to consider, i.e., $\dim(\Gamma)=2$ and $\dim(\Gamma)=1$.

If $\dim(\Gamma)=2$, then $\Gamma=\P^1\times\P^1$, and the canonical $F$-invariant probability measure on the projective complex manifold $\Gamma(\C)=\P^1(\C)\times\P^1(\C)$ is $\mu:=\mu_f\times\mu_{\overline{f}}$, where $\mu_f$ and $\mu_{\overline{f}}$ are the canonical probability measures on $\P^1(\C)$ which are $f$-invariant and $\overline{f}$-invariant, respectively. (Equivalently, $\mu_f$ and $\mu_{\overline{f}}$ are the unique maximal entropy measures of $f$ and $\overline{f}$, respectively.)

If $\dim(\Gamma)=1$, by the solution \cite{Ghioca2018c} of the dynamical Manin-Mumford problem for $F$ on $\P^1\times\P^1$, the irreducible curve $\Gamma$ is periodic under $F$. After replacing $f$ by $f^m$ for a suitable $m\in\Z_{>0}$, we may assume that the irreducible curve $\Gamma$ is $F$-invariant. Still denote by $\mu$ the canonical $F$-invariant probability measure on the projective complex analytic space $\Gamma(\C)$. Recall that if $g:X\to X$ is a polarized surjective endomorphism on a projective variety $X$ over $\C$ (hence $g$ is finite), then the canonical $g$-invariant probability measure $\mu_g$ on the projective complex analytic space $X(\C)$ is exactly the unique probability measure with continuous potentials on $X(\C)$ such that $g_*\mu_g=\mu_g$ and $g^*\mu_g=\deg(g)^{\dim(X)}\mu_g$. Since $f\circ\pi_1=\pi_1\circ F$ and
$\overline{f}\circ\pi_2=\pi_2\circ F$ are commuting finite morphisms, Tate's limiting construction of canonical measures implies that
$$\deg (\pi_1)\mu=\pi_1^*(\mu_f)\text{ and }\deg (\pi_2)\mu=\pi_2^*(\mu_{\overline{f}}).$$

Now we deal with both cases ($\dim(\Gamma)=2$ and $\dim(\Gamma)=1$) simultaneously. We will show that $2\sL_f<\int\log\lvert\det(dF)\rvert d\mu$.

For $n\in\Z_{>0}$, consider the probability measure $$\nu_n=\frac{1}{n_f(x_n)[K_n:K]}\sum_{j=0}^{n_f(x_n)-1}\sum_{\tau\in\Gal(K_n/K)}\delta_{F^j(\tau(p_n))},$$ where $K_n$ is the Galois closure of $K(x_n)$ over $K$ in $\overline{\Q}$. 

\textbf{Claim:} $\nu_n$ converges weakly to $\mu$ as $n\to+\infty$.

We prove the claim using Theorem \ref{equid}. Let $\sL$ be the ample line bundle
$$\pi_1^*\sO_{\P^1}(1)\otimes\pi_2^*\sO_{\P^1}(1)$$
on $\Gamma$. Then $(F\vert_{\Gamma})^*\sL\cong\sL^{\otimes d}$. By \cite[\S 2]{Zhang1995}, there exists a unique semipositive metric on $\sL$ making $(F\vert_{\Gamma})^*\sL\cong\sL^{\otimes d}$ an isometry; denote this metrized line bundle by $\overline{\sL}$. The canonical measure $\mu_{\overline{\sL}}$ associated with the metrized line bundle $\overline{\sL}$ on $\Gamma(\C)$ is exactly the canonical $F$-invariant probability measure $\mu$ on $\Gamma(\C)$, since the metric of ${\overline{\sL}}$ can be obtained by Tate's limit process, similar to $\mu$; see \cite[\S 2]{Zhang1995} and \cite[\S 3.5]{Yuan2008}. We need to verify the conditions (1) (small) and (2) (generic) in Theorem \ref{equid}. Condition (1) is trivial because $h_{\overline{\sL}}(\Gamma)=0$ and the height of every periodic algebraic point relative to $\overline{\sL}$ is zero in the setting of a polarized algebraic dynamical system (cf.\ \cite{Yuan2008}). For the condition (2), let $V$ be an arbitrary proper subvariety of $\Gamma$. After replacing $V$ by its Galois orbit, we may assume that $V$ is defined over $K$. Then the generic condition (2) follows immediately from Lemma \ref{pergene}. Thus the claim holds.

Let $n\in\Z_{>0}$, take $m\in\Z_{>0}$ such that $\lvert\rho_f(x_n)\rvert^m\in K$ by the assumption, and write $l=n_f(x_n)$. For every $\tau\in\Gal(\overline{\Q}/K)$ and $0\leq j\leq l-1$, we have
\begin{align*}
		&\det(dF^l(F^j(\tau(p_n))))^m=\det(dF^l(\tau(p_n)))^m\\
		=&\tau(\det(dF^l(p_n))^m)=\tau(\rho_f(x_n)^m\rho_{\overline{f}}(\overline{x_n})^m)\\
		=&\tau(\lvert \rho_f(x_n)\rvert^{2m})\\
		=&\lvert \rho_f(x_n)\rvert^{2m},
\end{align*}
where the first equality holds because both $\tau(p_n)$ and $F^j(\tau(p_n))$ belong to the same $F$-periodic orbit of length $l$, and $dF^l=d(F^l)$ in the expressions. Hence $\lvert\det(dF^l(F^j(\tau(p_n))))\rvert=\lvert \rho_f(x_n)\rvert^2$. Then by the definition of $\nu_n$, we have
\begin{align*}
		&\int\log\lvert\det(dF)\rvert d\nu_n=\frac{1}{l}\int\log\lvert\det(dF^l)\rvert d\nu_n\\
		=&\frac{1}{l^2[K_n:K]}\sum_{j=0}^{l-1}\sum_{\tau\in\Gal(K_n/K)}\log\lvert\det(dF^l(F^j(\tau(p_n))))\rvert\\
		=&\frac{1}{l^2[K_n:K]}\sum_{j=0}^{l-1}\sum_{\tau\in\Gal(K_n/K)}\log\lvert \rho_f(x_n)\rvert^2\\
		=&\frac{2}{l}\log\lvert \rho_f(x_n)\rvert=2\chi_f(x_n).
\end{align*}
For every $A\in\R$, since the function $\max\{\log\lvert\det(dF)\rvert,A\}$ is continuous, we have
\begin{align*}
	&2a=2\lim\limits_{n\to\infty}\chi_f(x_n)=\lim\limits_{n\to\infty}\int\log\lvert\det(dF)\rvert d\nu_n\\
	\leq&\lim\limits_{n\to\infty}\int\max\{\log\lvert\det(dF)\rvert,A\} d\nu_n\\
	=&\int\max\{\log\lvert\det(dF)\rvert,A\} d\mu.
\end{align*}
Letting $A\to-\infty$, by the monotone convergence theorem, we have
\begin{equation}\label{3.5}
2\sL_f<2a\leq\int\log\lvert\det(dF)\rvert d\mu.
\end{equation}

We now show that \eqref{3.5} leads to a contradiction by considering the two cases separately.

If $\dim(\Gamma)=2$, then
$$\int\log\lvert\det(dF)\rvert d\mu=\int_{\P^1\times\P^1}\log\lvert\det(d(f\times\overline{f}))\rvert d(\mu_f\times\mu_{\overline{f}})=\sL_f+\sL_{\overline{f}}=2\sL_f,$$
contradicting \eqref{3.5}.

If $\dim(\Gamma)=1$, then
\begin{align*}
&\int\log\lvert\det(dF)\rvert d\mu=\int_{\Gamma}\log\lvert\det(d(f\times\overline{f}))\rvert d\mu\\
=&\int_{\Gamma}\log\lvert\det(d f)\circ\pi_1\rvert d\mu+\int_{\Gamma}\log\lvert\det(d\overline{f})\circ\pi_2\rvert d\mu\\
=&\int_{\Gamma}\log\lvert\det(df)\circ\pi_1\rvert d\frac{\pi_1^*(\mu_f)}{\deg(\pi_1)}+\int_{\Gamma}\log\lvert\det(d\overline{f})\circ\pi_2\rvert d\frac{\pi_2^*(\mu_{\overline{f}})}{\deg(\pi_2)}\\
=&\int_{\P^1}\log\lvert\det(df)\rvert d\mu_f+\int_{\P^1}\log\lvert\det(d\overline{f})\rvert d\mu_{\overline{f}}\\
=&\sL_f+\sL_{\overline{f}}=2\sL_f,
\end{align*}
contradicting \eqref{3.5}.

Therefore, $f$ must be exceptional. This completes the proof. 
\end{proof}
\begin{proof}[Proof of Theorem \ref{thmmaindim} when $f$ is defined over $\overline{\Q}$]
	Assume that $f$ is defined over $\overline{\Q}$, hence $f$ is defined over some number field $K$. Suppose that Theorem \ref{thmmaindim} fails for $f$. Let $V$ be the $\Q$-span of $\chi_f(\Per^*(f))$ in $\R$; then $\dim_{\Q}V<\infty$. We can take $M\in\Z_{>0}$ and $x_1,\dots,x_M\in\Per^*(f)$ such that $\chi_f(x_1),\dots,\chi_f(x_M)$ generate $V$ over $\Q$. By enlarging $K$, we may assume that $\lvert\rho_f(x_1)\rvert,\dots,\lvert\rho_f(x_M)\rvert\in K$. Then for every $z_0\in\Per^*(f)$, its characteristic exponent $\chi_f(z_0)$ is a $\Q$-linear combination of $\chi_f(x_1),\dots,\chi_f(x_M)$; in other words, there exist $n\in\Z_{>0}$ and $n_1,\dots,n_M\in\Z$ such that
	$$\lvert\rho_f(z_0)\rvert^n=\lvert\rho_f(x_1)\rvert^{n_1}\cdots\lvert\rho_f(x_M)\rvert^{n_M}\in K,$$
	which contradicts Theorem \ref{thmmainnumf} since $f$ is not exceptional by assumption.
\end{proof}
\begin{Rem}
We will complete the whole proof of Theorem \ref{thmmaindim} in \S \ref{section6.3}.

According to \cite{Douady1993}, PCF maps are defined over $\overline{\Q}$ within the moduli space $\sM_d$ of rational maps of degree $d$, except for the family of flexible Latt\`es maps. Therefore, the above proof of Theorem \ref{thmmaindim} for $f$ defined over $\overline{\Q}$ includes the special case where $f$ is PCF (and non-exceptional).
\end{Rem}

\section{Pseudo-linear algebra for field multiplication}\label{sectionlinear}
\subsection{Pseudo-linear algebra}
Let $V$ and $W$ be two $\R$-linear spaces. A pseudo-morphism $f:V\to W$ is a pair $(V_f,f)$ where $V_f$ is a linear subspace of $V$ and $f:V_f\to W$ is an $\R$-linear map. If $x\in V\setminus V_f$, we write $f(x)=\infty$. When $W=\R$, we say that $f$ is a pseudo-function.

Denote by $\PHom(V,W)$ the set of pseudo-morphisms from $V$ to $W$. For $f,g\in\PHom(V,W)$, we define $f+g$ to be the pair $(V_f\cap V_g,f\vert_{V_f\cap V_g}+g\vert_{V_f\cap V_g})$. We denote by $0$ the pair $(V,0)$. We have $0+f=f$ for all $f\in\PHom(V,W)$. Then $\PHom(V,W)$ is a commutative monoid under the operation $+$ with the zero element $0$. For every $a\in\R$, we define $af$ to be the pair $(V_f,af)$. Note that $f+(-f)=(V_f,0)$, which is not $0$ if $V_f\neq V$. We have a natural embedding $\Hom(V,W)\hookrightarrow\PHom(V,W)$.
 
For $f\in\PHom(U,V)$ and $g\in\PHom(V,W)$, we define their composition $g\circ f$ to be $(U_f\cap f^{-1}(V_g),g\circ f\vert_{U_f\cap f^{-1}(V_g)})\in\PHom(U,W)$. Observe that if $f(v)=\infty$, then $g\circ f(v)=\infty$.

Fix a subset $O$ of $V$.
\begin{Def}
A sequence $(f_i)_{i=1}^\infty$ in $\PHom(V,\R)$ is called an $O$-sequence if the following conditions hold:
\begin{enumerate}[(i)]
	\item $f_i(O)\subseteq\R_{\geq 0}\cup\{\infty\}$ for $i\geq1$;
	\item for every $\la\in O$, the set $\{i\geq1:f_i(\la)\neq 0\}$ is finite.
\end{enumerate}
Clearly, any infinite subsequence of an $O$-sequence is also an $O$-sequence.
\end{Def}
\begin{Def}
Let $(\la_i)_{i=1}^\infty$ be a sequence in $O$ and $(f_i)_{i=1}^\infty$ be a sequence in $\PHom(V,\R)$. We say that $((\la_i)_{i=1}^\infty,(f_i)_{i=1}^\infty) $ is an upper triangular $O$-system (resp. weak upper triangular $O$-system) if the following conditions hold:
\begin{enumerate}[(i)]
	\item $(f_i)_{i=1}^\infty$ is an $O$-sequence;
	\item $f_i(\la_i)\in\R_{>0}$ (resp. $f_i(\la_i)\in\R_{>0}\cup\{\infty\}$) for $i\geq1$;
	\item $f_j(\la_i)=0$ for $j>i\geq1$.
\end{enumerate}
Clearly, an upper triangular $O$-system is a weak upper triangular $O$-system.
\end{Def}
\begin{Lem}\label{weakindep}
Let $((\la_i)_{i=1}^\infty,(f_i)_{i=1}^\infty)$ be a weak upper triangular $O$-system. Then $(\la_i)_{i=1}^\infty$ are linearly independent over $\R$.
\end{Lem}
\begin{proof}
Since $f_1(\la_1)\neq0$, we have $\la_1\neq0$. It remains to show that for all $l\geq2$, $\la_l$ is not contained in ${\rm span}_{\R}\{\la_i:i\leq l-1\}$. Suppose otherwise that $\la_l=\sum_{i=1}^{l-1} a_i\la_i$ for some $l\geq2$ and $a_i\in\R$, $1\leq i\leq l-1$. Then $f_l(\la_l)=\sum_{i=1}^{l-1} a_i f_l(\la_i)=0$, contradicting our assumption.
\end{proof}
Let $\tau:V\to V$ be an involution (here we only assume that that $\tau^2=\mathrm{Id}$; $\tau$ may be the identity on $V$).
\begin{Lem}\label{invoweak}
Assume that $\tau(O)\subseteq O$. Let $((\la_i)_{i=1}^\infty,(f_i)_{i=1}^\infty)$ be an upper triangular $O$-system. Then there exists a strictly increasing sequence $(m_i)_{i=1}^\infty$ in $\Z_{>0}$ such that the pair $((\la_{m_i}+\tau(\la_{m_i}))_{i=1}^\infty,(f_{m_i})_{i=1}^\infty)$ is a weak upper triangular $O^\prime$-system, where $O^\prime=\{\la_{m_i}+\tau(\la_{m_i}):i\in\Z_{>0}\}$.
\end{Lem}
\begin{proof}
	It is clear that $(f_i)_{i=1}^\infty$ is also an $O^\prime$-sequence. We construct $(m_i)_{i=1}^\infty$ recursively. Set $m_1:=1$. As $\tau(\la_1)\in O$, we have $f_1(\tau(\la_1))\in\R_{\geq0}\cup\{\infty\}$. Since $f_1(\la_1)\in\R_{>0}$, we have
	$$f_{m_1}(\la_{m_1}+\tau(\la_{m_1}))=f_1(\la_1)+f_1(\tau(\la_1))\in\R_{>0} \cup\{\infty\}.$$
	Assume that we have constructed $m_1,\dots,m_l$ satisfying the conditions for weak upper triangular $O^\prime$-systems. Since $(f_i)_{i =1}^\infty$ is an $O$-sequence and $\tau(\la_1),\dots,\tau(\la_l)\in O$, there exists $m_{l+1}>m_l$ such that $f_{m_{l+1}}(\tau(\la_{m_i}))=0$ for all $1\leq i\leq l$. Then for all $i=1,\dots,l$, we have
	$$f_{m_{l+1}}(\la_{m_i}+\tau(\la_{m_i}))=f_{m_{l+1}}(\la_{m_i})+f_{m_{l+1}}(\tau(\la_{m_i}))=0.$$
	Moreover,
	$$f_{m_{l+1}}(\la_{m_{l+1}}+\tau(\la_{m_{l+1}}))=f_{m_{l+1}}(\la_{m_{l+1}})+f_{m_{l+1}}(\tau(\la_{m_{l+1}}))\in\R_{>0}\cup\{\infty\}$$
	and
	$$f_{m_i}(\la_{m_{l+1}}+\tau(\la_{m_{l+1}}))=f_{m_i}(\la_{m_{l+1}})+f_{m_i}(\tau(\la_{m_{l+1}}))\in\R_{\geq0} \cup\{\infty\}.$$
	This completes the proof.
\end{proof}
Combining Lemma \ref{invoweak} and Lemma \ref{weakindep}, we obtain:
\begin{Cor}\label{invoindep}
Assume that $\tau(O)\subseteq O$. Let $((\la_i)_{i=1}^\infty,(f_i)_{i=1}^\infty)$ be an upper triangular $O$-system. Then $\dim_{\R}{\rm span}_{\R}\{ 2^{-1}(\la_i+\tau(\la_i)):i\geq 1\}=\infty$.
\end{Cor}
\begin{Rem}\label{rmkanyorder}
The discussion in this subsection also applies when $\R$ is replaced by any ordered field $F$ (for example, $\Q$ with the usual order).
\end{Rem}

\subsection{Linear algebra for multiplication}\label{sec5.2}
For every field $k$ of characteristic $0$, denote by $\mu_k$ the subgroup of roots of unity in $k$. Let
$$\rog:k^*\to\D(k):=k^*/\mu_k$$
be the quotient map. Extend $\rog$ to a map $\rog:k\to\D(k)\cup\{\infty\}$ by sending $0$ to $\infty$. (The notation ``$\rog$'' is used by analogy with the classical $\log$ function.) The embedding $k\hookrightarrow\overline{k}$ induces a natural embedding $\D(k)\hookrightarrow\D(\overline{k})$ of multiplicative abelian groups. Set $\D(k)_\Q:=\D(k)\otimes_\Z\Q$, where $\D(k)$ is viewed as a multiplicative abelian group (hence a $\Z$-module). Then $\D(k)_{\Q}$ is the subspace of $\D(\overline{k})$ spanned by $\D(k)$ over $\Q$. Set $\D(k)_{\R}:=\D(k)\otimes_\Z\R$.

Let $A\subseteq k$ be an integral domain with ${\rm Frac}(A)=k$. Define
$$\D(A):=\rog (A\setminus\{0\})\subseteq\D(k),$$
which is a subsemigroup of $\D(k)$. For every prime ideal $p$ of $A$, the surjective projection $A\to A/p$ induces a surjective map
$$s_p:\D(A)\cup\{\infty\}\to\D(A/p)\cup\{\infty\}.$$
In fact, we may view $s_p$ as a pseudo-morphism
$$s_p:\D(k)_\R\to\D({\rm Frac}(A/p))_\R$$
with domain $V_{s_p}:=(A\setminus p)\otimes_\Z\R$.

\subsection{Norms}
Let $k$ be a field of characteristic $0$. For every finite extension $\widetilde{k}$ of $k$, denote by
$N_{\widetilde{k}/k}:\widetilde{k}\to k$ the norm map. The norm map induces a map
$$\mathbf{n}_k:\D(\overline{k})_{\Q}\to\D(k)_\Q,\rog(x)\mapsto[k^\prime:k]^{-1}\rog(N_{k^\prime/k}(x)),$$
where $k^\prime$ is any finite extension of $k$ containing $x$. One can check that $\mathbf{n}_k$ is well-defined and $\Q$-linear. We also denote by $\mathbf{n}_k$ its $\R$-linear extension $\mathbf{n}_k:\D(\overline{k})_\R\to\D(k)_\R$. When the field $k$ is clear, we write $\mathbf{n}$ for $\mathbf{n}_k$. 

\subsection{Valuations}
Assume that $K$ is a number field. Denote by $\sM_K$ the set of all places of $K$. For every $v\in\sM_K$, we define a map $\phi_v:\D(K)_\R\to\R$ by $\phi_v(\rog(x))=-\log(\lvert x\rvert_v)$ for $x\in K^*$ and extending $\R$-linearly, where $\left|\cdot\right|_v$ is a fixed multiplicative norm on $K$ corresponding to $v$. It is easy to check that the map $\phi_v$ is well-defined and $\R$-linear. We denote the restriction of $\phi_v$ to $\D(K)_\Q$ by $\phi_v:\D(K)_\Q\to\R$ as well. For every $a\in\D(K)_\R$, it is clear that the set $\{v\in\sM_K:\phi_v(a)\neq0\}$ is finite.

Let $S$ be a finite subset of $\sM_K$ containing all archimedean places. Let $\sO_{K, S}$ be the ring of $S$-integers in $K$. Let $\sO$ be the integral closure of $\sO_{K, S}$ in $\overline{K}$. For every $v\in\sM_K\setminus S$ and $\la\in\sO$, we have $\phi_v\circ\mathbf{n}_K(\la)\geq0$. Write $\sM_K\setminus S=\{v_1,v_2,\dots\}$. Then $(\phi_{v_i})_{i=1}^\infty\subseteq\Hom(\D(K)_\R,\R)$ is an $\sO_{K,S}$-sequence and $(\phi_{v_i}\circ\mathbf{n}_K)_{i=1}^\infty\subseteq\Hom(\D(\overline{K})_\R,\R)$ is an $\sO$-sequence.

\subsection{Complex conjugation and absolute value}
Denote by $\tau:\C\to\C$ the complex conjugation. Then $\R$ is the invariant subfield $\C^\tau$ of $\C$ under $\tau$. As $\Q$-vector spaces, we have an identification $$\D(\R)=\R^*/\{\pm1\}\to\R,\rog (a)\mapsto\log\lvert a\rvert,$$ where the latter $\log$ is the classical one on $\R_{>0}$. Using this identification, the absolute value map on $\C$ can be viewed as the norm map
$$\mathbf{n}_\R:\D(\C)\to\D(\R),\rog(x)\mapsto\frac{\rog(x)+\rog(\tau(x))}{2}.$$

Let $\mathbf{k}$ be an algebraically closed subfield of $\C$ stable under complex conjugation $\tau$. Still denote by $\tau\in\Gal(\mathbf{k}/\Q)$ the restriction of complex conjugation on $\mathbf{k}$, which is an involution of $\mathbf{k}$. Let $\mathbf{k}^\tau=\mathbf{k}\cap\R$ be the $\tau$-invariant subfield of $\mathbf{k}$. Then the restriction of the absolute value $\mathbf{n}_{\R}$ on $\mathbf{k}$ is
$$\mathbf{n}_{\mathbf{k}^\tau}:\D(\mathbf{k})\to\D(\mathbf{k}^\tau),\rog(x)\mapsto\frac{\rog(x)+\rog(\tau(x))}{2}.$$

We shall prove the following theorem involving an involution $\tau$ (which may not be complex conjugation and may be the identity) in the next section. 

\begin{Thm}\label{thminvo}
Let $\mathbf{k}$ be an algebraically closed field of characteristic $0$. Let $\tau\in\Gal(\mathbf{k}/\Q)$ be an element with $\tau^2=\mathrm{Id}$. If $f:\P^1_\mathbf{k}\to\P^1_\mathbf{k}$ is an endomorphism over $\mathbf{k}$ of degree at least $2$ which is not PCF, then the $\Q$-subspace of $\D(\mathbf{k}^\tau)_\Q$ spanned by $\{\mathbf{n}_{\mathbf{k}^\tau}(\rog(\rho_f(x))):x\in\Per^*(f)(\mathbf{k})\}$ is of infinite dimension.
\end{Thm}

Applying Theorem \ref{thminvo} to the case $\mathbf{k}=\C$ and $\tau$ being complex conjugation, we deduce Theorem \ref{thmmaindim} in the case where $f$ is not PCF.
\begin{Rem}
By setting $\tau=\mathrm{Id}$, Theorem \ref{thminvo} yields the following result:

Assume that $\mathbf{k}$ is an algebraically closed field of characteristic $0$. If $f:\P^1_\mathbf{k}\to\P^1_\mathbf{k}$ is an endomorphism over $\mathbf{k}$ of degree at least $2$ which is not PCF, then the $\Q$-subspace of $\D(\mathbf{k})_\Q$ spanned by $\{\rog(\rho_f(x)):x\in\Per^*(f)(\mathbf{k})\}$ is of infinite dimension.
\end{Rem}

\section{Proofs of Theorem \ref{thminvo} and Theorem \ref{thmmaindim} }\label{section6}
\subsection{Proof of Theorem \ref{thminvo}: the case $\mathbf{k}=\overline{\Q}$}\label{section6.1}
Assume that $\mathbf{k}=\overline{\Q}$ and let $\tau$ be an element in $\Gal(\overline{\Q}/\Q)$ with $\tau^2=\mathrm{Id}$.

Denote by $\sC_f$ the set of critical points of $f$. Since $f$ is not PCF, there exists $o\in\sC_f$ such that the (forward) $f$-orbit $O_f(o)$ of $o$ is infinite. Fix such a critical point $o$. Let $X$ be the union of all (forward) orbits of periodic critical points of $f$. Then $X$ is finite.

Pick a number field $K$ satisfying $\tau(K)=K$ such that $f$, $o$, and all points in $X$ are defined over $K$.

Denote by $\sM_K$ the set of all places of $K$. Take a finite subset $B\subseteq\sM_K$ containing all archimedean places, satisfying $\tau(B)=B$, and such that $f$ has good reduction at $v$ for every $v\in\sM_K\setminus B$. Then $\tau(\sO_{K, B})=\sO_{K, B}$. For $x\in\Per(f)(\overline{\Q})$, set $\la(x)=(n_f(x))^{-1}\rog(\rho_f(x))\in\D(\overline{\Q})_{\R}\cup\{\infty\}$. Recall that $n_f(x)\in\Z_{>0}$ is the exact period of $x$ and $\rho_f(x)=df^{n_f(x)}(x)\in\overline{\Q}$ is the multiplier of $x$. For all $x\in\Per^*(f)(\overline{\Q})$, we have $\mathbf{n}_K (\la(x))\in O:=\D(\sO_{K, B})\otimes_{\Z}\Q_{>0}$ by \cite[Corollary 2.23 (a)]{Silverman2007}. Note that $O=\tau(O)$ is invariant under $\tau$, where $\tau$ descends to $\D(K)$ and is extended $\Q$-linearly to $\D(K)_\Q$.

Denote by $\C_v$ the completion of the algebraic closure of $K_v$ for $v\in\sM_K$. Every embedding $\sigma:\overline{\Q}\hookrightarrow\C_v$ gives a bijection $\sigma:\Per(f)(\overline{\Q})\to\Per(f)(\C_v)$. Observe that for every $x\in\Per(f)(\overline{\Q})$, we have $\sigma(\la(x))=\la(\sigma(x))$.

For every $v\in\sM_K\setminus B$ and $x\in\P^1(\C_v)$, denote by $\tilde{x}\in\P^1(\overline{\widetilde{K_v}})$ the reduction of $x$ in the special fiber at $v$ and $f_v:\P^1_{\overline{\widetilde{K_v}}}\to\P^1_{\overline{\widetilde{K_v}}}$ the reduction of $f$. (Here, $\overline{\widetilde{K_v}}$ is a fixed algebraic closure of the residue field $\widetilde{K_v}$ of $K_v$.) After enlarging $B$, we may assume that $\tilde{o}\notin X_v$ for every $v\in\sM_K\setminus B$, where $X_v$ is the reduction of $X$ contained in $\P^1(\widetilde{K_v})\subseteq\P^1(\overline{\widetilde{K_v}})$.

Observe that for every $x\in\Per(f)(\overline{\Q})$ of exact period $n\geq1$ and every $v\in\sM_K\setminus B$, we have $\phi_v(\mathbf{n}_K(\la(x)))\geq0$ by \cite[Corollary 2.23 (a)]{Silverman2007}. Moreover, the following statements are equivalent by Hensel's lemma:

({\romannumeral 1}) $\phi_v(\mathbf{n}_K(\la(x)))>0$;

({\romannumeral 2}) there exists an embedding $\sigma:\overline{\Q}\hookrightarrow\C_v$ such that $(f^n_v)^\prime(\widetilde{\sigma(x)})=0$;

({\romannumeral 3}) there exist an embedding $\sigma:\overline{\Q}\hookrightarrow\C_v$, $q\in\sC_f$, and $m\in\Z_{\geq0}$, such that $\widetilde{\sigma(q)}$ is $f_v$-periodic with $\widetilde{\sigma(x)}=f_v^m(\widetilde{\sigma(q)})$.

For every $v\in\sM_K\setminus B$, denote by $P_v$ the union of all $f_v$-orbits of periodic critical points of $f_v$. The set $P_v$ is finite. For every $v\in\sM_K\setminus B$ and $q\in P_v$, by Hensel's lemma there exists a unique periodic point $y=y(q)\in\Per(f)(\C_v\cap\overline{\Q})$ such that $\tilde{y}=q$. Then there exists a unique $\Gal(\overline{\Q}/K)$-orbit $O(q)$ in $\P^1(\overline{\Q})$ such that for some (hence every) $x\in O(q)$, there exists an embedding $\sigma:\overline{\Q}\hookrightarrow\C_v$ satisfying $\widetilde{\sigma(x)}=q$ (here $O(q)$ is the $\Gal(\overline{\Q}/K)$-orbit of $y=y(q)$). In particular, we have $X_v\subseteq P_v$ and $\cup_{q\in X_v}O(q)=X$. It follows that the set
$$Q_v:=\{x\in\Per(f)(\overline{\Q}):\phi_v(\mathbf{n}_K(\la(x)))>0\}=\bigcup_{q\in P_v}O(q)\subseteq\Per(f)(\overline{\Q})$$
is finite and $X\subseteq Q_v$. Moreover, $Q_v=X$ if and only if $P_v=X_v$.
\begin{Lem}\label{lem5.1}
	The set $S:=\{v\in\sM_K\setminus B:P_v\setminus X_v\neq\emptyset\}$ is infinite.
\end{Lem}
\begin{proof}
	Recall that $o\in\sC_f$ is a critical point of $f$ that is not $f$-preperiodic, and we assume that $\tilde{o}\notin X_v$ for every $v\in\sM_K\setminus B$. By \cite[Lemma 4.1]{Benedetto2012}, there are infinitely many $v\in\sM_K\setminus B$ such that there exists $n\in\Z_{>0}$ with $f_v^n(\tilde{o})=\tilde{o}$. For such a place $v$, we have $\tilde{o}\in P_v\setminus X_v$, which proves the lemma.
\end{proof}
\begin{Lem}\label{lem5.2}
	There exist a sequence $(x_i)_{i=1}^\infty$ in $\Per^*(f)(\overline{\Q})$ and a sequence $(v_i)_{i=1}^\infty$ in $\sM_K\setminus B$ such that $\phi_{v_i}(\mathbf{n}(\la(x_i)))>0$ for all $i\geq1$ and $\phi_{v_j}(\mathbf{n}(\la(x_i)))=0$ for $j\neq i$. In particular, $((\mathbf{n}(\la(x_i)))_{i=1}^\infty,(\phi_{v_i})_{i=1}^\infty)$ is an upper triangular $O$-system for $\D(K)_\Q$.
\end{Lem}
\begin{proof}
We construct these sequences recursively.
	
By Lemma \ref{lem5.1}, $S$ is infinite. Pick $v_1\in S$, then there exists
$$x_1\in Q_{v_1}\setminus X\subseteq\Per^*(f)(\overline{\Q}).$$
We have $\phi_{v_1}(\mathbf{n}(\la(x_1)))>0$.

Assume that we have constructed $x_1,\dots,x_m\in\Per^*(f)(\overline{\Q})$ and $v_1,\dots,v_m\in\sM_K\setminus B$ such that $\phi_{v_j}(\mathbf{n}(\la(x_i)))\geq 0$ and equality holds if and only if $j\neq i$. The set $\cup_{i=1}^m Q_{v_i}\setminus X$ is finite. Then there exists a finite set $T_m\subseteq\sM_K$ such that for all $x\in\cup_{i=1}^m Q_{v_i}\setminus X$ and $v\in\sM_K\setminus T_m$, we have $\phi_v(\mathbf{n}(\la(x)))=0$. By Lemma \ref{lem5.1}, we can take a place $v_{m+1}\in S\setminus(\{v_1,\dots,v_m\}\cup T_m)$. Then $\phi_{v_{m+1}}(\mathbf{n}(\la(x_i)))=0$ for all $1\leq i\leq m$. Pick $x_{m+1}\in Q_{v_{m+1}}\setminus X$. We have $\phi_{v_{m+1}}(\mathbf{n}(\la(x_{m+1})))>0$. It follows that $x_{m+1}\notin\cup_{i=1}^m Q_{v_i}$. Thus $\phi_{v_i}(\mathbf{n}(\la(x_{m+1})))=0$ for all $1\leq i\leq m$. We conclude the proof of Lemma \ref{lem5.2}.
\end{proof}
The conclusion now follows from Corollary \ref{invoindep} and Remark \ref{rmkanyorder}.\qed

\subsection{Proof of Theorem \ref{thminvo}: the general case}\label{section6.2}
	Denote by $\sC_f$ the set of critical points of $f$. Since $f$ is not PCF, there exists $o\in\sC_f$ which is not $f$-preperiodic. Fix such a critical point $o$. Fix a subfield $K$ of $\mathbf{k}$ such that $K/\Q$ is finitely generated, $\tau(K)=K$, and $o,f$ are defined over $K$. Without loss of generality, we may assume that $\mathbf{k}=\overline{K}$.
	
	Take a finitely generated $\Z$-subalgebra $A$ of $K$ with ${\rm Frac}(A)=K$ and $\tau(A)=A$. After shrinking $\Spec(A)$, we may assume that there exists an endomorphism $f_A:\P^1_A\to\P^1_A$ over $A$ whose restriction $f_K:\P^1_K\to\P^1_K$ on the generic fiber $\P^1_K$ satisfies $f=f_K\otimes_K\mathbf{k}$.
	
	For every $c\in\Spec(A\otimes_\Z\overline{\Q})(\overline{\Q})$, denote by $f_c$ the specialization of $f_A$ at $c$, and $o_c$ the specialization of $o$ at $c$. Then $o_c$ is a critical point of $f_c$. By \cite[Lemma 3.3]{Ghioca2018}, we can take $c\in\Spec(A\otimes_\Z\overline{\Q})(\overline{\Q})$ such that the $f_c$-orbit of $o_c$ is infinite, i.e., $o_c$ is not $f_c$-preperiodic. In particular, the specialization map $f_c$ is not PCF. Take a number field $L\subseteq\overline{K}$ such that $c$ is defined over $L$ and $\tau(L)=L$. Denote by $A_1$ the $\Z$-algebra generated by $A$, $\sO_L$, and $\tau(\sO_L)$ in $A\otimes_{\Z}\overline{\Q}$. We may replace $\Spec(A)$ by a suitable non-empty affine Zariski open subset of $\Spec(A_1)$ such that $f_A:\P^1_A\to\P^1_A$ remains everywhere well-defined and $c\in\Spec(A\otimes_{\sO_L}L)(L)$ after the replacement. We view $\Spec(A)$ as an $\sO_L$-scheme. After shrinking $\Spec(A)$, the Zariski closure of $c$ in $\Spec(A)$ is isomorphic to $\Spec(\sO_{L,S})$ for a finite set $S\subseteq\sM_L$ of places of $L$ containing all archimedean places, which corresponds to a prime ideal $p$ of $A$.
	
	Denote by $s_p:\D(K)_\R\dashrightarrow\D(L)_\R$ the pseudo-morphism associated with $p$ as in \S \ref{sec5.2}. We have $s_p(\D(A))\subseteq\D(\sO_{L,S})\cup\{\infty\}$. Then for every $v\in\sM_L\setminus S$ and $\la\in\D(A)$, we have $\phi_v(s_p(\la))\in\R_{\geq0}\cup\{\infty\}$. Moreover, for every $\la\in\D(A)$ with $s_p(\la)\neq\infty$, there are only finitely many $v\in\sM_L\setminus S$ such that $\phi_v(s_p(\la))\neq0$.
	
	For every $y\in\Per(f)(\overline{K})$, denote by $y_c$ the set of $x\in\Per(f_c)(\overline{L})$ whose image is contained in the Zariski closure of the image of $y$ in $\P_A^1$. For every $y\in\Per(f)(\overline{K})$, $y_c$ is finite and nonempty. On the other hand, for every $x\in\Per(f_c)(\overline{L})$, the set $\{y\in\Per(f)(\overline{K}):x\in y_c\}$ is finite and nonempty. Moreover, if $x\in y_c$, then
	$$s_p(\mathbf{n}_K(\la(y)))=\mathbf{n}_L(\la(x)).$$
	
	Since the set of $x\in\Per(f_c)(\overline{L})$ with $\mathbf{n}_L(\la(x))=\infty$ is finite, the set
	$$W_c:=\{y\in\Per(f)(\overline{K}):s_p(\mathbf{n}_K(\la(y)))=\infty\}$$
	is also finite. Similarly, $W_{\tau(c)}:=\{y\in\Per(f)(\overline{K}):s_{\tau(p)}(\mathbf{n}_K(\la(y)))=\infty\}$ is finite.
	
	By Lemma \ref{lem5.2}, there exist sequences $(x_i)_{i=1}^\infty$ in $\Per(f_c)(\overline{L})$ and $(v_i)_{i=1}^\infty$ in $\sM_L\setminus S$ such that $((\mathbf{n}_L(\la(x_i)))_{i=1}^\infty,(\phi_{v_i})_{i=1}^\infty)$ is an upper triangular $(\D(\sO_{L,S})\otimes_\Z\Q_{>0})$-system for $\D(L)_{\Q}$. For every $i\in\Z_{>0}$, we can take $y_i\in\Per(f)(\overline{K})$ such that $x_i\in(y_i)_c$, and hence
	$$s_p(\mathbf{n}_K(\la(y_i)))=\mathbf{n}_L(\la(x_i)).$$
	After removing finitely many terms, we may assume that $y_i\notin W_c\cup W_{\tau(c)}$ for all $i\geq1$. It follows that $\mathbf{n}_K(\la(y_i))\in\rog(A\setminus(p\cup\tau(p)))\otimes_\Z\Q_{>0}$ for $i\geq1$. Observe that $(\phi_{v_i}\circ s_p)_{i=1}^\infty$ is a $(\rog(A\setminus(p\cup\tau(p)))\otimes_\Z\Q_{>0})$-sequence. It follows that
	$$((\mathbf{n}_K(\la(y_i)))_{i=1}^\infty,(\phi_{v_i}\circ s_p)_{i=1}^\infty)$$
	is an upper triangular $(\rog(A\setminus(p\cup\tau(p)))\otimes_\Z\Q_{>0})$-system for $\D(K)_\Q$. Since $$\rog(A\setminus(p\cup\tau(p)))\otimes_\Z\Q_{>0}$$ is invariant under $\tau$, we conclude the proof by Corollary \ref{invoindep} and Remark \ref{rmkanyorder}.\qed
	
\subsection{Proof of Theorem \ref{thmmaindim}}\label{section6.3}
	There are two cases: 
	
	1. The case where $f$ is PCF. According to \cite{Douady1993}, PCF maps are defined over $\overline{\Q}$ within the moduli space $\sM_d$ of rational maps of degree $d$, except for the family of flexible Latt\`es maps. Since $f$ is non-exceptional, it is defined over $\overline{\Q}$, and Theorem \ref{thmmaindim} was already proved in \S \ref{section4}.
	
	2. The case where $f$ is not PCF. Then Theorem \ref{thmmaindim} is a direct consequence of Theorem \ref{thminvo}.
	
	This completes the proof of Theorem \ref{thmmaindim}.\qed
	
\section{Proofs of the Applications}
\subsection{Proof of Theorem \ref{thm1.8}}
Recall that our goal is to prove the Zariski-dense orbit conjecture for $f=f_1\times\cdots\times f_N:X=\P^1_k\times\cdots\times\P^1_k\to X$, where $f_1,\dots,f_N:\P^1_k\to\P^1_k$ are non-constant endomorphisms on the projective line over an algebraically closed field $k$ of characteristic $0$.

Without loss of generality, we may assume that $k$ is of finite transcendence degree over $\Q$. Fix an embedding of $k$ into $\C$. We view $f$ as an endomorphism on $X$ defined over $\C$. According to \cite[Theorem 3.34]{Xie2022}, we may assume that all $f_j:\P^1\to\P^1$ have degree at least $2$ for $1\leq j\leq N$.

Assume first that all $f_j$ are not exceptional, $1\leq j\leq N$. Corollary \ref{cor1.6} implies that we can take $x_j\in\Per^*(f_j)(\C)$ for $1\leq j\leq N$ such that the multipliers $\rho_{f_1}(x_1),\dots,\rho_{f_N}(x_N)$ are multiplicatively independent in $\C$. After replacing $f$ by an iterate, we may assume that $f_j(x_j)=x_j$ for $1\leq j\leq N$, and the multipliers $(\rho_{f_j}(x_j)=f_j^\prime(x_j))_{j=1}^N$ remain multiplicatively independent. Let $x=(x_1,\dots,x_N)\in X(k)$. Then $x$ is a fixed point of $f$ (smooth in the fixed locus of $f$) such that the eigenvalues of $df\vert_x$ are nonzero and multiplicatively independent. The conclusion then follows from \cite{Amerik2011}.

If all $f_j$ are exceptional, $1\leq j\leq N$, this case is relatively easy, and we refer to \cite[Theorem 4.1, Theorem 1.14, and Lemma 3.30]{Xie2022} (see also \cite[\S 9.3]{Xie2022}).

Now assume that $0\leq s\leq N$ such that $f_1,\dots,f_s$ are not exceptional and $f_{s+1},\dots,f_N$ are exceptional. Let $l(f)=\min\{s,N-s\}\geq0$. The case $l(f)=0$ has just been treated. An induction on $l(f)N\in\Z_{\geq0}$ then proves Theorem \ref{thm1.8}, as shown in the last several paragraphs of \cite[\S 9.3]{Xie2022}.\qed

\subsection{Proof of Theorem \ref{pcfdec}}
Recall that $f:\P^1(\C)\to\P^1(\C)$ is a rational map of degree $\geq2$. We aim to study the condition: (1) $f$ is PCF. Our goal is to prove that (1) is equivalent to each of the conditions (2) and (3) in Theorem \ref{pcfdec}, which are determined by the multiplier spectrum and the length spectrum of $f$, respectively.
\medskip

Using the terminology and notations in \S \ref{sectionlinear}, it is clear that (2) and (3) are equivalent to the following conditions $(2)^\prime$ and $(3)^\prime$, respectively.

$(2)^\prime$ $\rho_f(x)\in\overline{\Q}$ for all $x\in\Per(f)(\C)$ and the $\Q$-subspace of $\D(\Q)_\Q$ generated by $\mathbf{n}_{\Q}(\rog(\rho_f(x)))$ for $x\in\Per^*(f)(\C)$ is of finite dimension over $\Q$.

$(3)^\prime$ $\lvert\rho_f(x)\rvert\in\overline{\Q}$ for all $x\in\Per(f)(\C)$ and the $\Q$-subspace of $\D(\Q)_\Q$ generated by $\mathbf{n}_\Q(\rog(\lvert\rho_f(x)\lvert))$ for $x\in\Per^*(f)(\C)$ is of finite dimension over $\Q$.

\medskip

We now prove that the three conditions (1), $(2)^\prime$, and $(3)^\prime$ are equivalent.

\medskip
(1) $\Rightarrow(2)^\prime$ and $(3)^\prime$:

Suppose that $f$ is PCF. By \cite{Douady1993}, PCF maps are defined over $\overline{\Q}$ in $\sM_d$, except for the family of flexible Latt\`es maps. If $f$ is flexible Latt\`es, then according to \cite[Lemma 5.6]{milnor2006lattes}, we have $\rho_f(x)\in\Z$ for all $x\in\Per(f)(\C)$. If $f$ is defined over $\overline{\Q}$, then clearly $\rho_f(x)\in\overline{\Q}$ for all $x\in\Per(f)(\C)=\Per(f)(\Q)$. Thus, in both cases, $\rho_f(x),\lvert\rho_f(x)\rvert\in\overline{\Q}$ for all $x\in\Per(f)(\C)$.

Suppose that $(2)^\prime$ is false, i.e.,
$$\dim_\Q{\rm span}_\Q\{\mathbf{n}_{\Q}(\rog(\rho_f(x))):x\in\Per^*(f)(\C)\}=\infty.$$
By \cite[Corollary 3.9]{milnor2006lattes}, $f$ cannot be a flexible Latt\`es map, so $f$ is defined over $\overline{\Q}$, and hence over some number field $K$. We use the notation and ideas from \S \ref{section6.1} with $\tau=\mathrm{Id}$. Take a finite subset $B\subseteq\sM_K$ containing all archimedean places such that $f$ has good reduction at $v$ for every $v\in\sM_K\setminus B$. For every $v\in\sM_K\setminus B$, the reduction $f_v$ is still PCF, and its critical orbits come from those of $f$. Arguing similarly as in \S \ref{section6.1}, it is easy to see that the set $$\mathcal{W}:=\{x\in\Per^*(f)(\C):\phi_v(\mathbf{n}_K(\la(x)))=0,\forall v\in\sM_K\setminus B\}$$ is co-finite in $\Per^*(f)(\C)$, i.e., $\Per^*(f)(\C)\setminus\mathcal{W}$ is finite. It is well-known that ${\rm rank}(\sO_{K,B}^\times)=(\#B)-1<\infty$ (cf. \cite[Theorem 3.12]{Narkiewicz2004}). Note that
$$\mathbf{n}_K(\rog(\rho_f(x)))\in\D(\sO_{K,B})$$
for all $x\in\Per^*(f)(\C)$ by \cite[Corollary 2.23 (a)]{Silverman2007}. We deduce that
$$\dim_\Q{\rm span}_\Q\{\mathbf{n}_K(\rog(\rho_f(x))):x\in\Per^*(f)(\C)\}\leq\#(\Per^*(f)(\C)\setminus\mathcal{W})+(\#B)-1$$
is finite, which implies $\dim_\Q{\rm span}_\Q\{\mathbf{n}_\Q(\rog(\rho_f(x))):x\in\Per^*(f)(\C)\}<\infty$, contradicting the assumption. Thus $(2)^\prime$ must hold.

$(3)^\prime$ follows from a similar argument corresponding to the case where $\tau$ is complex conjugation in \S \ref{section6.1}.

\medskip
$(2)^\prime$ or $(3)^\prime\Rightarrow$ (1):

Let $\tau$ be the identity on $\C$ in the case $(2)^\prime$ and complex conjugation in the case $(3)^\prime$. We prove $(2)^\prime$ $\Rightarrow (1)$ and $(3)^\prime$ $\Rightarrow(1)$ simultaneously. Suppose $f$ is not PCF. In particular, $f$ is not a flexible Latt\`es map. By $(2)^\prime$ (or $(3)^\prime$) and Theorem \ref{thmalglenintro}, the map $f$ is defined over $\overline{\Q}$, hence over some number field $K$. We use the notation and ideas from \S \ref{section6.1} for $\tau$. Arguing similarly as in \S \ref{section6.1}, after enlarging the number field $K$, we can choose a finite set $B\subseteq\sM_K$ of places containing all archimedean places satisfying the following conditions:
\begin{itemize}
	\item $K/\Q$ is Galois and $\tau(K)=K$, so $\tau|_K\in\Gal(K/\Q)$;
	\item $f$ has good reduction at $v$ for every $v\in\sM_K\setminus B$;
	\item $B$ is invariant under every $\sigma\in\Gal(K/\Q)$;
	\item $\tilde{o}\notin X_v$ for every $v\in\sM_K\setminus B$, where $o$ is a critical point of $f$ that is not $f$-preperiodic and $X$ is the union of all $f$-orbits of periodic critical points of $f$. (Here $o$ and all points in $X$ are assumed to be defined over $K$.)
\end{itemize}
A slight modification of the proof of Lemma \ref{lem5.2} shows that there exist a sequence $(x_i)_{i=1}^\infty$ in $\Per^*(f)(\overline{\Q})$ and a sequence $(v_i)_{i=1}^\infty$ in $\sM_K\setminus B$ satisfying:
\begin{align}
	\label{7.15}&\phi_{v_i}(\mathbf{n}_K(\la(x_i)))>0\text{ for all }i\geq1;\\
	\label{7.16}&\phi_{\sigma(v_j)}(\mathbf{n}_K(\la(x_i)))=0\text{ for all }i\neq j\text{ and }\sigma\in\Gal(K/\Q).
\end{align}
For $i\geq1$, let $p_i$ be the rational prime below $v_i$ and $\widetilde{B}\subseteq\sM_{\Q}$ the set of places of $\Q$ lying below the places in $B$. From \eqref{7.15} and \eqref{7.16}, it is easy to see that the pair $((\mathbf{n}_{\Q}(\la(x_i)))_{i=1}^\infty,(\phi_{p_i})_{i=1}^\infty)$ is an upper triangular $(\D(\sO_{\Q,\widetilde{B}})\otimes_\Z\Q_{>0})$-system for $\D(\Q)_\Q$. Note that $\D(\sO_{\Q,\widetilde{B}})\otimes_\Z\Q_{>0}$ is $\tau$-invariant, where $\tau$ descends to $\D(\Q)$ and is extended $\Q$-linearly to $\D(\Q)_\Q$. By Corollary \ref{invoindep} and Remark \ref{rmkanyorder}, this contradicts $(2)^\prime$ and $(3)^\prime$, respectively. Thus, $f$ must be PCF. \qed

\subsection*{Acknowledgement}
The second-named author Junyi Xie would like to thank Thomas Gauthier, Gabriel Vigny, Charles Favre, and Serge Cantat for helpful discussions. The authors would like to thank the anonymous referee for many thoughtful comments, which help to improve the manuscript. The first-named author would like to thank Beijing International Center for Mathematical Research in Peking University for the invitation. The first-named author, Zhuchao Ji, is supported by National Key R\&D Program of China (No.2025YFA1018300), NSFC Grant (No.12401106), and ZPNSF grant (No.XHD24A0201). The second- and third-named authors, Junyi Xie and Geng-Rui Zhang, are supported by NSFC Grant (No.12271007).

\subsection*{Declarations}
\subsubsection*{Data availability}
Data sharing is not applicable to this article because no data were generated or analyzed during the current study.
\subsubsection*{Conflict of interest}
The authors declare that they have no conflict of interest.

\bibliography{b}

\begin{thebibliography}{10}

\bibitem{Amerik2011}
Ekaterina Amerik, Fedor Bogomolov, and Marat Rovinsky.
\newblock Remarks on endomorphisms and rational points.
\newblock {\em Compos. Math.}, 147:1819--1842, 2011.

\bibitem{Amerik2008}
Ekaterina Amerik and Fr{\'e}d{\'e}ric Campana.
\newblock Fibrations m\'eromorphes sur certaines vari\'et\'es \`a fibr\'e
  canonique trivial.
\newblock {\em Pure Appl. Math. Q.}, 4(2):509--545, 2008.

\bibitem{Benedetto2012}
Robert~L. Benedetto, Dragos Ghioca, P{\"a}r Kurlberg, and Thomas~J. Tucker.
\newblock A case of the dynamical {M}ordell-{L}ang conjecture.
\newblock {\em Math. Ann.}, 352(1):1--26, 2012.
\newblock With an appendix by Umberto Zannier.

\bibitem{Bosch1990}
Siegfried Bosch, Werner L\"{u}tkebohmert, and Michel Raynaud.
\newblock {\em N\'{e}ron models}, volume~21 of {\em Ergeb. Math. Grenzgeb.
  (3)}.
\newblock Springer, Berlin, 1990.

\bibitem{Buff2022}
Xavier Buff, Thomas Gauthier, Valentin Huguin, and Jasmin Raissy.
\newblock Entire or rational maps with integer multipliers.
\newblock arXiv:2212.03661v2, 2023. To appear in {\em Algebraic, Complex, and
  Arithmetic Dynamics}, {\em Simons Symp.}, Springer, Cham, 2026.

\bibitem{Douady1993}
Adrien Douady and John~H. Hubbard.
\newblock A proof of {T}hurston's topological characterization of rational
  functions.
\newblock {\em Acta Math.}, 171(2):263--297, 1993.

\bibitem{FLM83}
Alexandre Freire, Artur Lopes, and Ricardo Ma{\~{n}}{\'e}.
\newblock An invariant measure for rational maps.
\newblock {\em Bol. Soc. Bras. Mat.}, 14:45--62, 1983.

\bibitem{Ghioca2018c}
Dragos Ghioca, Khoa~D. Nguyen, and Hexi Ye.
\newblock The dynamical {M}anin-{M}umford conjecture and the dynamical
  {B}ogomolov conjecture for endomorphisms of {$(\Bbb P^1)^n$}.
\newblock {\em Compos. Math.}, 154(7):1441--1472, 2018.

\bibitem{Ghioca2018}
Dragos Ghioca and Junyi Xie.
\newblock The {D}ynamical {M}ordell--{L}ang {C}onjecture for {S}kew--{L}inear
  {S}elf-{M}aps. {A}ppendix by {M}ichael {W}ibmer.
\newblock {\em Int. Math. Res. Not. IMRN}, 2020(21):7433--7453, 2020.

\bibitem{Huguin2023}
Valentin Huguin.
\newblock Rational maps with rational multipliers.
\newblock {\em J. \'{E}c. polytech. Math.}, 10:591--599, 2023.

\bibitem{ji2023dao}
Zhuchao Ji and Junyi Xie.
\newblock D{AO} for curves.
\newblock arXiv:2302.02583v2, 2023.

\bibitem{Ji2023}
Zhuchao Ji and Junyi Xie.
\newblock Homoclinic orbits, multiplier spectrum and rigidity theorems in
  complex dynamics.
\newblock {\em Forum Math. Pi}, 11:Paper No. e11, 37 p., 2023.

\bibitem{Levy2014}
Alon Levy.
\newblock {AIM} workshop {P}ostcritically finite maps in complex and arithmetic
  dynamics, 2014.

\bibitem{Lyubich83}
Mikhail~Ju. Lyubich.
\newblock Entropy properties of rational endomorphisms of the {R}iemann sphere.
\newblock {\em Ergodic Theory Dynam. Systems}, 3:351--385, 1983.

\bibitem{Mane83}
Ricardo Ma{\~{n}}{\'e}.
\newblock On the uniqueness of the maximizing measure for rational maps.
\newblock {\em Bol. Soc. Bras. Mat.}, 14:27--43, 1983.

\bibitem{McMullen1987}
Curt McMullen.
\newblock Families of rational maps and iterative root-finding algorithms.
\newblock {\em Ann. of Math. (2)}, 125(3):467--493, 1987.

\bibitem{Medvdev}
Alice Medvedev and Thomas Scanlon.
\newblock Invariant varieties for polynomial dynamical systems.
\newblock {\em Ann. of Math. (2)}, 179(1):81--177, 2014.

\bibitem{Milnor}
John Milnor.
\newblock {\em Dynamics in one complex variable}, volume 160 of {\em Ann. of
  Math. Stud.}
\newblock Princeton Univ. Press, Princeton, NJ, 3rd edition, 2006.

\bibitem{milnor2006lattes}
John Milnor.
\newblock On {L}att{\`e}s maps.
\newblock In {\em Dynamics on the {Riemann} sphere: {A} {Bodil} {Branner}
  {Festschrift}}, pages 9--43. Eur. Math. Soc., Z{\"u}rich, 2006.

\bibitem{MoriwakiAra}
Atsushi Moriwaki.
\newblock {\em Arakelov geometry}, volume 244 of {\em Transl. Math. Monogr.}
\newblock Amer. Math. Soc., Providence, RI, 2014.
\newblock Translated from the 2008 Japanese original.

\bibitem{Narkiewicz2004}
W{\l}adys{\l}aw Narkiewicz.
\newblock {\em Elementary and analytic theory of algebraic numbers}.
\newblock Springer Monogr. Math. Springer, Berlin, 3rd edition, 2004.

\bibitem{Pakovich2023}
Fedor Pakovich.
\newblock Invariant curves for endomorphisms of {$\Bbb P^1 \times \Bbb P^1$}.
\newblock {\em Math. Ann.}, 385(1-2):259--307, 2023.

\bibitem{Poonen2017}
Bjorn Poonen.
\newblock {\em Rational points on varieties}, volume 186 of {\em Grad. Stud.
  Math.}
\newblock Amer. Math. Soc., Providence, RI, 2017.

\bibitem{Silverman2007}
Joseph~H. Silverman.
\newblock {\em The arithmetic of dynamical systems}, volume 241 of {\em Grad.
  Texts in Math.}
\newblock Springer, New York, NY, 2007.

\bibitem{Silverman2012}
Joseph~H. Silverman.
\newblock {\em Moduli spaces and arithmetic dynamics}, volume~30 of {\em CRM
  Monogr. Ser.}
\newblock Amer. Math. Soc., Providence, RI, 2012.

\bibitem{Tucker2014}
Thomas Tucker.
\newblock Problem 6 in the problem list of the {AIM} workshop {P}ostcritically
  finite maps in complex and arithmetic dynamics, 2014.
\newblock
  \url{http://admin.aimath.org/resources/postcritically-finite-maps-in-complex-and-arithmetic-dynamics/problemlist/}.

\bibitem{Xie2017}
Junyi Xie.
\newblock The existence of {Z}ariski dense orbits for polynomial endomorphisms
  of the affine plane.
\newblock {\em Compos. Math.}, 153(8):1658--1672, 2017.

\bibitem{Xie2023}
Junyi Xie.
\newblock Remarks on algebraic dynamics in positive characteristic.
\newblock {\em J. Reine Angew. Math.}, 797:117--153, 2023.

\bibitem{Xie2022}
Junyi Xie.
\newblock The existence of {Z}ariski dense orbits for endomorphisms of
  projective surfaces (with an appendix in collaboration with {T}homas
  {T}ucker).
\newblock {\em J. Amer. Math. Soc.}, 38:1--62, 2025.

\bibitem{Partial1}
Junyi Xie and Xinyi Yuan.
\newblock {P}artial heights, entire curves, and the geometric {B}ombieri-{L}ang
  conjecture.
\newblock arXiv:2305.14789v2, 2023.

\bibitem{Partial2}
Junyi Xie and Xinyi Yuan.
\newblock {T}he geometric {B}ombieri-{L}ang conjecture for ramified covers of
  abelian varieties.
\newblock arXiv:2308.08117, 2023.

\bibitem{Yuan2008}
Xinyi Yuan.
\newblock Big line bundles over arithmetic varieties.
\newblock {\em Invent. Math.}, 173(3):603--649, 2008.

\bibitem{YZ21}
Xinyi Yuan and Shou-Wu Zhang.
\newblock {\em Adelic line bundles on quasi-projective varieties}, volume 221
  of {\em Ann. of Math. Stud.}
\newblock Princeton Univ. Press, Princeton, NJ, 2026.

\bibitem{Zdunik90}
Anna Zdunik.
\newblock Parabolic orbifolds and the dimension of the maximal measure for
  rational maps.
\newblock {\em Invent. Math.}, 99(3):627--649, 1990.

\bibitem{zdunik2014characteristic}
Anna Zdunik.
\newblock Characteristic exponents of rational functions.
\newblock {\em Bull. Pol. Acad. Sci. Math.}, 62(3):257--263, 2014.

\bibitem{Zhang1995}
Shouwu Zhang.
\newblock Small points and adelic metrics.
\newblock {\em J. Algebraic Geom.}, 4(2):281--300, 1995.

\end{thebibliography}

\end{document}